\documentclass[reqno,centertags,11pt]{amsart}
\usepackage{amsmath,amsthm,amscd,amssymb} \usepackage{latexsym}
\usepackage[utf8]{inputenc}
\usepackage{graphicx}
\usepackage{color}

 \newcommand{\R}{{\mathbb{R}}}

 \newcommand{\Z}{{\mathbb{Z}}}

\newcommand{\beq}{\begin{equation}}
\newcommand{\eeq}{\end{equation}}
\newcommand{\bdm}{\begin{displaymath}}
\newcommand{\edm}{\end{displaymath}} \newcommand{\ba}{\begin{align}}
\newcommand{\ea}{\end{align}} \newcommand{\bpf}{\begin{proof}}
\newcommand{\epf}{\end{proof}}

 \allowdisplaybreaks

\newtheorem{definition}{Definition}

\newtheorem{theorem}{Theorem}

\newtheorem{remark}[theorem]{Remark}
\newtheorem{corollary}[theorem]{Corollary}
\newtheorem{lemma}[theorem]{Lemma}

\usepackage{hyperref}

\begin{document}

\title{Resonant Solutions of the Non-linear
 Schr\"{o}dinger 
Equation with Periodic Potential}

\author[A.~ Duaibes, Yu.~Karpeshina]{Arein Duaibes, Yulia Karpeshina}

\address{Department of Mathematics, College of Saint Benedict and Saint John's University, 2850 Abbey Plaza
Collegeville, Minnesota 56321, USA}
\email{aduaibes001@csbsju.edu}

\address{Department of Mathematics, University of Alabama at Birmingham, University Hall, Room 4005,
1402 10th Avenue South,
Birmingham AL 35294-1241, USA}
\email{karpeshi@uab.edu}%

\date{\today}

\begin{abstract}
The goal is a construction of stationary solutions  close to  non-trivial combinations of two plane waves at high energies for  a periodic  non-linear
 Schr\"{o}dinger 
equation in dimension two. The corresponding isoenergetic surfaces are described for every sufficiently large energy.
 \end{abstract}

\maketitle


\section{Introduction}

Non-linear Schr\"{o}dinger equation 
\begin{equation}
iu_{t}=-\Delta u+Vu+\sigma |u|^2u   \label{1a}
\end{equation}
describes a variety of physical phenomena in optics, acoustics, quantum condensate (Gross-Pitaevskii equation), hydrodynamics, plasma physics, etc. The case of a periodic potential $V$ is of great interest. The equation has been studied for a long time.
 However, majority of the studies is in physics literature and concerns with one-dimensional situation. Higher dimensions are investigated mostly numerically, (e.g. \cite{KS02}  -- \cite{YD03}) for  dimension two). Theoretical papers for  periodic  multidimensional situations are more sparse. We definitely have to refer here
to \cite{B1}-\cite{CLY} devoted  to  periodic initial value problems with $V$ being zero or an operator of multiplication in the Fourier dual space. We approach  \eqref{1a}  somewhat differently. 
We are interested in stationary solutions of \eqref{1a} in multiple dimensions. It was proven in \cite{KKS1}, \cite {KKS2} that there are stationary solutions of \eqref{1a} close to  plane waves $e^{i(\vec k, \vec x)}$  for an extensive set of $\vec k$ in dimensions two and three. In this paper we show that  there exist stationary solutions  close to  non-trivial combinations of two plane waves at high energies for  a periodic  non-linear
 Schroedinger 
equation in dimension two. The corresponding isoenergetic surfaces are described.

We start with considering a nonlinear polyharmonic equation: 
\begin{equation}\label{main equation}
(-\Delta)^{l}u(\vec{x})+V(\vec{x})u(\vec{x})+\sigma |u(\vec{x})|^{2}u(\vec{x})=\lambda u(\vec{x}),~\vec{x}\in \mathbb{R}^{n},
\end{equation}
and the quasi-periodic boundary conditions:
\begin{equation}\label{main condition}
\begin{cases}
~u(x_{1},\cdots,x_{s}+2\pi,\cdots,x_{n})=e^{2\pi it_{s}}u(x_{1},\cdots,x_{s},\cdots,x_{n}),\\
~\frac{\partial}{\partial x_{s}}u(x_{1},\cdots,x_{s}+2\pi,\cdots,x_{n})=e^{2\pi it_{s}}\frac{\partial}{\partial x_{s}}u(x_{1},\cdots,x_{s},\cdots,x_{n}),\\
\vdots\\
~\frac{\partial^{2l-1}}{\partial x_{s}^{2l-1}}u(x_{1},\cdots,x_{s}+2\pi,\cdots,x_{n})=e^{2\pi it_{s}}\frac{\partial^{2l-1}}{\partial x_{s}^{2l-1}}u(x_{1},\cdots,x_{s},\cdots,x_{n}),\\
~~s=1,\cdots,n.
\end{cases}
\end{equation}
 where $l$ is an integer such that $2l>n$ or $l=1$, $n=2$ and $\sigma \in \R$. We consider a periodic potential $V(\vec{x})$. We  assume that the potential $V$ is a trigonometric polynomial with the elementary cell of periods $Q=[0,2\pi]^n$:
 \begin{equation} \label{V}
V(\vec{x})=\sum_{q\neq0,|q|\leq R_0}v_qe^{i\langle{q,\vec{x}}\rangle}, \ \ \ 0<R_0<\infty,
\end{equation}
\begin{align*}
\int_{Q} V(\vec{x})d\vec{x}=0.
\end{align*} 
The last assumption can be done without loss of generality. We require that $V$ is not identically zero.
 We consider  the case $l=1$, $n=2$ for a sufficiently small $V$ : $\|V\|<\varepsilon _* $, here $\varepsilon _*\neq \varepsilon _*(\lambda )$.


We start with discussing the linear operator:
\begin{equation}\label{schrodinger eq}
H=(-\Delta)^l + V(\vec{x})
\end{equation}
in $L^2(\mathbb{R}^n)$.
The spectral analysis of the operator $H$  can be reduced to the analysis of operators $H(\vec{t})$, $\vec{t}\in K$, with $K=[0,1)^n$ being the elementary cell of dual lattice. The operators $H(\vec{t})$ are defined by \eqref{schrodinger eq} and the quasiperiodic conditions \eqref{main condition}
%
in $L^2(Q)$. 
%
%

It is well-known that $H(\vec{t})$ has a discrete semi-bounded from below spectrum $\bigcup_{n=1}^{\infty}\lambda_n(\vec{t})$.
By Bloch's theorem \cite{RS}, the following is true:\\
    1. The spectrum of $H$ has band structure:
\begin{align*}
    \Delta=\bigcup_{n=1}^{\infty}\bigcup_{\vec{t} \in K}\lambda_n(\vec{t}).
\end{align*}
2. A complete system of generalized eigenfunctions of $H$ can be obtained by extending  eigenfunctions of $H(\vec{t})$ quasiperiodically to the whole space $\mathbb{R}^n$. 

We use the notation $H_0(\vec{t})$ for  $H(\vec{t})$ with $V=0$, its  eigenvalues being given by
\begin{align}\label{p_j(t)}
p_j^{2l}(\vec{t})=|\vec{P}_j(\vec{t})|^{2l},
\end{align}
where $$\vec{P}_j(\vec{t})=\vec{P}_j(\vec{0})+\vec{t}=2\pi j+\vec{t},\ j\in \mathbb{Z}^n,\ \vec{t}\in K.$$ The corresponding eigenfunction is a plane wave 
\begin{align}\label{v(x)}
e^{i\langle{\vec{P}_j(\vec{t}),\vec{x}}\rangle}
\end{align}
and the corresponding spectral projection is $E_j$ defined by
\begin{align}\label{def of E_j}
(E_j)_{rm}=\delta_{jr}\delta_{jm}
\end{align}
in the basis (\ref{v(x)}) in $L^2(Q)$.
Obviously any $\vec{k}\in \mathbb{R}^n$ can be uniquely written in the form $\vec{k}=\vec{P}_j(\vec{t})$ for some $j\in \mathbb{Z}^n$ and $\vec{t}\in K$. 
An isoenergetic surface $S_0(k)$ of $H_0$ is a set of $t\in K$ such that $p_j(t)=k$ for some $j\in Z^n$. It looks like a sphere "packed" into K.
Namely, 
\begin{equation}
S_0=\mathcal{K}S(k), \label{S0}
\end{equation}
where $S(k)$ is the sphere in $\mathbb{R}^n$ centered at the origin with radius $k$ and 
\begin{equation} \label{map K}
\mathcal{K}: \mathbb{R}^n \rightarrow K, ~\mbox{  } \mathcal{K}\vec{P}_q(\vec{t}) = \vec{t}.
\end{equation}
 The process of obtaining $S_0(k)$ starts by dividing $S(k)$ by the dual lattice $\{\vec{P}_m(0)\}_{m\in\mathbb{Z}^n}$ into pieces and all the pieces are then translated into $K$. 
%

 Perturbation series for  eigenvalues and their spectral projections for $H(t)$ with respect to the operator $H_0(\vec{t})$ 
 are obtained in \cite{K97}. When $2l>n$,   perturbation formulas are valid for large $k$ and a set  $\chi_0(k,\delta)\subset S_0(k)$ of $\vec{t}$, such that $p_j(\vec{t})=k$  for some $j$ and 
\begin{align}\label{chi0 cond}
    \underset{q\neq0}{\min}|p_j^{2l}(\vec{t})-p_{j+q}^{2l}(\vec{t})|>k^{2l-n-\delta},
\end{align}
 $0<\delta<2l-n$. For $n=2,\ l=1$,  some additional inequalities are needed.  The inequality \eqref{chi0 cond} means that $\vec{t}$ is sufficiently far from self-intersections of $S_0(k)$.
The set $\chi_0(k,\delta)$ has an asymptotically full measure on the isoenergetic surface $S_0(k)$ as $k\to \infty$.
It is proven
 that for every $\vec{t} \in \chi_0(k,\delta)$, the operator $H(t)$ has an  eigenvalue close to $p_{j}^{2l}(\vec{t})=\mid\vec{P}_{j}(\vec{t})\mid^{2l}$ with an eigenfunction close to $e^{i\langle{ \vec{P}_{j}(\vec{t}), \vec{x} }\rangle}$. The exact formulas of the eigenvalue and its spectral projection are given in Section 2, Theorem \ref{linear pert 2l>n}. 
In \cite{KKS1} we proved an analog of Theorem  \ref{linear pert 2l>n} for the non-linear case \eqref{main equation}, \eqref{main condition}; see  Theorem \ref{Kim}. We call \eqref{chi0 cond} the non-resonant case for $t$. Correspondingly, $\chi_0(k,\delta)$ is the non-resonant set.
%

 The vector $\vec{k}$ is said to satisfy the von Laue Diffraction condition if
\begin{align}\label{von laue}
    |\vec{k}|=|\vec{k}-\vec{P}_q(\vec{0})|,
\end{align}
for some $q\in \mathbb{Z}^n \setminus \{0\}$. If $\vec t:\vec P_j(\vec t)=\vec k$, then,  $\vec t$ obviously belongs to a self-intersection of $S_0(k)$, Therefore, $\vec t \not  \in  \chi_0(k,\delta)$.  Perturbation formulas for eigenvalues and spectral projections of $H(\vec t)$ with respect to $H_0(\vec t)$ don't work in this situation.  The situation when $\vec k$ is in a vicinity of \eqref{von laue} is called the resonant case.
It turns out that in a vicinity of  \eqref{von laue} the eigenvalues and its  spectral projections can be approximated by those  of a model matrix
\begin{equation}\begin{pmatrix}
p_j^{2l}(\vec{t}) & v_{q}\\
v_{-{q}} & p_{j-{q}}^{2l}(\vec{t})
\end{pmatrix}.  \label{2x2} \end{equation}
Further, we assume:
\begin{equation} q:v_q\neq 0 \label{vq}. \end{equation} 
We denote its eigenvalues and eigenvectors by $\hat \lambda ^{\pm}$ and $\hat e^{\pm}$, correspondingly.
In \cite{K97} we described a set $\chi_q(k,\delta)\subset S_0(k)\setminus \chi_0(k,\delta)$ of quasimomenta for which  formulas for eigenvalues and spectral projections of $H(\vec t)$ were constructed for every sufficiently large $k$.  Indeed, let set $\mathcal{K}S_q(k,n-2+\delta)$ be the image of the spherical layer $S_q(k,n-2+\delta)$ under the map $\mathcal{K}$ given by \eqref{map K}.
Here,
\begin{align} \label{S-q}
S_q(k,n-2+\delta)=\{\vec{x}\in S(k):\big||\vec{x}|^2-|\vec{x}-\vec{P}_q(0)|^2\big| <4k^{-n+2-\delta}\}.
\end{align}
In other words, we consider $S_q(k,n-2+\delta)$ being a neighborhood of (\ref{von laue}) in the sphere $S(k)$ and then shift it into the cube $K$ using (\ref{map K}). 
In Section 2 here, we explicitly describe the set $\chi_q(k,\delta)$ which has an asymptotically full measure on $\mathcal{K}S_q(k,n-2+\delta)$. For $\vec t$ in a small vicinity of $\chi_q(k,\delta)$
we give formulas for eigenvalues  and spectral projections of $H(\vec t)$. The corresponding eigenfunctions are close to  non-trivial combinations of two plane waves $e^{i<\vec k,\vec x >}$ and  $ e^{i<\vec k-\vec{P}_q(\vec{0}),\vec x >}$. The coefficients of plane waves are described through eigenvectors $\hat e^{\pm }$ of \eqref{2x2}. Thus, we have a pair of eigenfunctions for every $\vec k $. The set  $\chi_q(k,\delta)$ is called the resonant set. The exact statement of the result  is given in Section 2, Theorem \ref{singular set main result}. It is proven that the isoenergetic surface corresponding to the pair of  solutions  is close to two pieces of isoenergetic surface of  \eqref{2x2}, see Corollary \ref{17}.

\textbf{Here, we will construct formulas for solutions $u^{\pm}, \lambda ^{\pm}$ of \eqref{main equation}, \eqref{main condition} when $\vec{k}$ is close to (\ref{von laue})}. 
Essentially, we make use of the results proved for the linear case, and apply a method of successive  approximation. 
Namely, we consider the part of equation (\ref{main equation}) with the nonlinear term, i.e., $V+\sigma|u^{\pm}|^2$, as an unknown potential. 
For the method of successive approximation, we define two maps, $\hat{\mathcal{M}^{\pm}}$, and construct two sequences  of potentials, $W_m^{\pm}$ converging to  the potentials $W ^{\pm}=V+\sigma|u^{\pm}|^2$. 

The map $\hat{\mathcal{M}^{\pm}}$ is defined in Section 3.1.
Then, the sequence of potentials $\{W_m^{\pm}\}_{m=0}^{\infty}$ is constructed as follows:
\begin{align}
  W_0^{\pm}= &V,\\. \label{sequence}
  W_m^{\pm}=&\hat{\mathcal{M}}^{\pm}W^{\pm}_{m-1}.
\end{align}
It turns out that it is a Cauchy sequence converging to a function $W$ with respect to a norm $\|\cdot\|_*$, 
\begin{align}\label{norm ||_*}
\|W\|_*=\sum_{j\in\mathbb{Z}^n}|w_j|,  
\end{align}
where $w_j$'s are the Fourier coefficients of $W$.
Namely, we show that 
\begin{align}\label{ch1 estimate}
    \|W^{\pm}-W^{\pm}_{m}\|_* \leq   (ck^{-\gamma})^{m+1},
\end{align}
for some $\gamma>0$.
Then, we show that  convergence of $\{W_m^{\pm}\}_{m=0}^{\infty}$ to $W^{\pm}$ leads to  convergence of the sequences of the spectral projections $\{\hat{E}_m^{\pm}(\vec{t})\}_{m=0}^{\infty}$ and their
 corresponding eigenvalues $\{\hat{\lambda}_m^{\pm}(\vec{t})\}_{m=0}^{\infty}$ to  $\hat{E}_W^{\pm}(\vec{t})$ and   $\hat{\lambda}^{\pm}_W(\vec{t})$, here $\hat{E}^{\pm}_m(\vec{t})$, $\hat{\lambda}^{\pm}_m(\vec{t})$ and  $\hat{E}^{\pm}_W(\vec{t})$, $\hat{\lambda}^{\pm}_W(\vec{t})$ are spectral projections and their eigenvalues of  $H_0+W_m^{\pm}$ and $H_0+W^{\pm}$.  Corresponding to  $\hat{E}^{\pm}_W(\vec{t})$ eigenfunction $u^{\pm}\equiv \hat u_{W}^{\pm}$, 
 see \eqref{def of u eq-e}, \eqref{def of psi}, solves \eqref{main equation}, \eqref{main condition}. It is shown that  $\ u^{\pm}$ is  close to a combination of two plane waves  $e^{i<\vec k,\vec x >}$ and  $ e^{i<\vec k-\vec{P}_q(\vec{0}),\vec x >}$ under some restriction on its amplitude. The coefficients of plane waves are described by eigenvectors 	$\hat e^{\pm}$ of \eqref{2x2} with a good accuracy.  The exact statement of the result  is given  by Theorem \ref{mainmain}.
 The sequences  $\{W_m^{\pm}\}_{m=0}^{\infty}$, $\{\hat{E}^{\pm}_m(\vec{t})\}_{m=0}^{\infty}$ and $\{\hat{\lambda}^{\pm}_m(\vec{t})\}_{m=0}^{\infty}$.  can be differentiated with respect to $\vec t$ and maintain their convergence. As a result the asymptotic formulas 
 \eqref{T(m) of diff of spectral proj}, \eqref{nabla lambda} for $\nabla \hat{E}^{\pm}_W(\vec{t})$, $\nabla \hat{\lambda}^{\pm}_W(\vec{t})$ are proven, and  $|\nabla \hat{\lambda}^{\pm}_W(\vec{t})|\approx 2lk^{2l-1}$. It follows that 
 the surface $ \lambda ^{\pm}(\vec{t})=\lambda _0 $ is in the $C(V)\lambda _0^{-\hat \gamma }$-neighborhood of  $ \hat{\lambda}^{\pm}(\vec t)=\lambda _0$ for every sufficiently large $\lambda _0$, here  $\hat{\lambda}^{\pm}$ are eigenvalues of \eqref{2x2} and $\hat \gamma >(2l-1)(2l)^{-1}$, see Theorem \ref{25} and Corollary \ref{25a}.

In Section 4 we consider the physically interesting case $l=1$, $n=2$. All considerations of the previous case can be done for $l=1$, $n=2$ under the assumption that $V$ is sufficiently small $\|V\|_*<\epsilon ^9$, $0<\epsilon <\epsilon _0$, $\epsilon _0 \neq \epsilon _0(\lambda )$, for the waves with a sufficiently small amplitude $A$, the restriction on $A$ being given in terms of  $\epsilon $, see \eqref{restrictions}.

 The isoenergetic surface $\lambda^{\pm}_{}(\vec{t})=\lambda _0 $ is in a small -neighborhood of  $ \hat{\lambda}^{\pm}(\vec t)=\lambda _0$.
 
If  $V$ is not small, then there are solutions  close to  non-trivial combinations of more than two plane waves. In this case the model operator is described not by \eqref{2x2}, but by a Hill operator. This result will be proven in a forthcoming paper.

The paper is organized as follows. Preliminary results are described in Section 2,  Section 2.1 containing the results for a linear case \cite{K97} and Section 2.2 containing the results for the non-linear equation, non-resonant case \cite{KKS1}. Chapter 3 is devoted to proving the main result for the case $2l>n$. We consider $l=1, n=2$ in Chapter 4.\\

{\bf Acknowledgement.} The authors are thankful to G\"{u}nter Stolz for usefull discussions.

\section{Preliminary Results}
In this chapter, we present a brief review of previous results needed for proof of the main result of this paper. 


  
\subsection{Linear  Polyharmonic  Equation with Periodic Potential}
Let us consider an operator in $L_2(Q)$ given by the differential expression:
\begin{equation}\label{operator H}
H_\alpha (\vec t)u=(-\Delta)^{l}u+\alpha Vu, 
\end{equation}
with the quasi-periodic boundary conditions \eqref{main condition}. Here
 $l$ is an integer such that $2l>n$, 
$-1\leq\alpha\leq1$. 

Perturbation series for eigenvalues and its spectral projections for $H_{\alpha }(t)$ are constructed on a nonsingular set $\chi_0$ \cite{K97}. We include here a construction of $\chi_0(k,\delta)$ and discuss the perturbation theory for $H_{\alpha }(\vec t)$.

%


\begin{lemma}\label{difference eigenvalues 2l>n lemma} For an arbitrarily small positive, $\delta$, $2\delta<2l-n$, and sufficiently large $k>k_0(\delta )$,  there exists a non-resonant set $\chi _0(k,\delta )$, belonging to the isoenergetic surface $S_{0}(k)$ of the free operator $H_{0}(t)$, such that, for any $\vec{t}$ in it,
\begin{align}
&1.~\mbox{there exists a unique $m\in \mathbb{Z}^{n}$ such that $p_{m}^{2l}(\vec{t})=k^{2l}$},\\
&2.~\min_{j\neq m}|p_{j}^{2l}(\vec{t})-p_{m}^{2l}(\vec{t})|>k^{2l-n-\delta}.
\end{align}
Moreover, the nonsingular set has an asymptotically full measure on $S_{0}(k)$:
\begin{equation*} 
\frac{s\big(S_{0}(k)\setminus\chi _0(k,\delta )\big)}{s\big(S_{0}(k)\big)}=O(k^{-\delta/2}), ~\mbox{as $k \to \infty$}, 
\end{equation*}
\\~\mbox{where $s(.)$ denotes the Lebesgue measure}.
\end{lemma}

\begin{corollary}\label{lemma t 2l>n} Suppose $\vec t$ belongs to the $(k^{-n+1-2\delta })$-neighborhood in $K$ 
of the resonant set $\chi _0(k,\delta )$,  $0<2\delta <2l-n$. Then for all $z $ lying on the circle $C_{0}=\{z \in \mathbb{C}: |z-k^{2l}|=k^{2l-n-\delta}\}$ and any $i$ of $\mathbb{Z}^{n},$ the inequality 
$$2\mid p_{i}^{2l}(\vec{t})-z\mid>k^{2l-n-\delta}$$
holds.
\end{corollary}
 Let us introduce the functions $g_r(k,\vec{t})$ and the operator-valued functions $G_r(k,\vec{t})$, $r=0$, $1$, $...$, for $\vec{t} \in \chi_0(k,\delta)$:
\begin{equation} \label{g_r}
g_{r}(k,\vec t)=\frac{(-1)^{r}}{2\pi ir}\mbox{Tr} \oint_{C_{0}}\left((H_{0}(\vec t)-z)^{-1}V\right)^{r}dz,
\end{equation}
\begin{equation}\label{G_r 2l>n}
G_{r}(k,\vec{t})=\frac{(-1)^{r+1}}{2\pi i}\oint_{C_{0}} \left((H_{0}(\vec{t})-z)^{-1}V\right)^{r}(H_{0}(\vec{t})-z)^{-1}dz.
\end{equation}
The following theorem presents the main result for \eqref{operator H}.
\begin{theorem}\label{linear pert 2l>n} Suppose $\vec{t}$ belongs to the $(k^{-n+1-2\delta })$-neighborhood in $K$ 
of the 
non-resonant set $\chi _0(k,\delta )$,  $0<2\delta <2l-n$. Then for sufficiently 
large $k$,  $k>k_0(\|V\|,\delta )$,  
there exists a single eigenvalue of the operator $H(\vec{t})$ 
 in
the interval 
$\varepsilon (k,\delta )\equiv (k^{2l}-k^{2l-n-\delta },
k^{2l}+k^{2l-n-\delta })$. 
It is given by the series
\begin{equation} \label{ev 2l>n}
\lambda_{j} (\alpha,\vec{t})=p_j^{2l}(\vec{t})+\sum _{r=2}^{\infty }{\alpha}^{r}g_{r}(k,\vec{t}), 
\end{equation}
converging absolutely in the disk $|\alpha|\leq 1$, where the index $j$
is uniquely determined from the relation 
$p_j^{2l}(\vec{t})\in \varepsilon (k,\delta )$ and the spectral projection, corresponding to $\lambda_j(\alpha,\vec{t})$ is given by 
the series
\begin{equation} \label{proj op 2l>n}
E_j(\alpha,\vec{t})=E_j+ \sum_{r=1}^{\infty}{\alpha}^{r}G_{r}(k,\vec{t}),
\end{equation}
which converges in the trace class $\bf{S_1}$ uniformly with respect to $\alpha$ in the disk $|\alpha|\leq 1$.

Moreover, for the coefficients $g_{r}(k,\vec{t})$ and  $G_{r}(k,\vec{t})$, the following estimates hold.
\begin{align} \label{ev c es 2l>n}
|g_{r}(k,\vec{t})|<k^{2l-n-\delta}k^{-(2l-n-2\delta)r},
\end{align}
\begin{align}\label{proj op c es 2l>n}
\|G_{r}(k,\vec{t})\|_{1}\leq k^{-(2l-n-2\delta)r}.
\end{align}
\end{theorem}
The series (\ref{ev 2l>n}) and (\ref{proj op 2l>n})  are differentiable termwise with respect to $\vec{t}$  
in the ($k^{-n+1-2\delta}$)-neighborhood in $\mathbb{C}^n$ of the set $\chi_0(k,\delta)$, see  \cite{K97}. Indeed, let
%
\begin{align}\label{T(m)}
T(m)\equiv \frac{\partial ^{\mid m\mid }}{\partial t_1^{m_1}
\partial t_2^{m_2}...\partial t_n^{m_n}}, \mbox{ where} \mid m\mid \equiv m_1+m_2+...+m_n,
\end{align}
\begin{align}
\ m!\equiv m_1!m_2!...m_n!, \ \ \ \ 0\leq \mid m\mid <\infty,\ \ \ \   \ T(0)f\equiv f.
\end{align}
%

%
%
%

%
%
\begin{theorem} \label{Thm T(m), nonlin (4l>n+1)} 
Under the conditions of  Theorem \ref{linear pert 2l>n} 
the series  (\ref{ev 2l>n}) and (\ref{proj op 2l>n})  can be differentiated with respect to $\vec t$ any number of times, and 
they retain their asymptotic character. The coefficients $g_r(k,\vec{t})$ and $G_r(k,\vec{t})$ satisfy the following estimates in the $(k^{-n+1-2\delta})$-neighborhood in $\mathbb{C}^n$ of the nonsingular set $\chi_0(k,\delta)$:
\begin{equation}
\mid T(m)g_r(k,\vec t)\mid <m!k^{-(2l-n-2\delta)(r-1)}k^{\mid m\mid (n-1+2\delta )} 
\end{equation}
\begin{equation}
\| T(m)G_r(k,\vec t)\|_{1}<m!k^{-(2l-n-2\delta)r}k^{\mid m\mid (n-1+2\delta )}. 
\end{equation}
\end{theorem}
\begin{corollary} \label{T(m) estimate of coeff, nonlin, 4l>n+1} 
The following estimates hold for the perturbed eigenvalue and its
spectral projection.
\begin{equation}
\mid T(m)(\lambda _j(\alpha, \vec t)-p_j^{2l}(\vec t))\mid < cm!k^{-(2l-n-2\delta)}k^{\mid m\mid (n-1+2\delta )} 
\end{equation}
\begin{equation} 
\| T(m)(E_j(\alpha ,\vec t)-E_j)\|_{1} < cm!k^{-(2l-n-2\delta)}k^{\mid m\mid (n-1+2\delta )} 
\end{equation}
In particular,
\begin{equation} 
\mid \lambda_{j}(\alpha, \vec{t}))-p^{2l}_j(\vec{t})\mid < ck^{-(2l-n-2\delta)}
\end{equation}
\begin{equation} 
\| E_{j}(\alpha ,\vec t)-E_j)\|_{1} < ck^{-(2l-n-2\delta)}
\end{equation}
\begin{equation} 
\mid \nabla _{\vec t} \lambda_{j}(\alpha, \vec{t})-2l\vec{P}_j(\vec{t})p^{2l-2}_j(\vec{t})\mid < ck^{-(2l-n-2\delta)+n-1+2\delta }.\end{equation}
\end{corollary}

\begin{corollary} The surface $ \lambda_{j}(\alpha,\vec{t})=\lambda _0 $ is in the real $\lambda _0^{-(4l-n+1-2\delta )}$-neighborhood of  $ \chi _0 (k,\delta )$ for every sufficiently large $\lambda _0$. \end{corollary}

%

Next, we  consider  formulas for $\vec{t} \in S_0(k)\setminus \chi_0(k,\delta)$. 
%
This means that there is at least one $q \in \mathbb{Z}^n$ such that 
\begin{align*}
|p_{j-q}^{2l}(\vec{t})-p_{j}^{2l}(\vec{t})|<k^{2l-n-\delta}, ~\vec{P}_{j}(\vec{t})\in S(k).
\end{align*}
Below, we present  the main results and the perturbation formulas for an eigenvalue and its spectral projection for $\vec{t}$ belonging to a "resonant set", $\chi_q(k,\delta)\subset  S_0(k)\setminus \chi_0(k,\delta)$. We  set $\alpha=1$ for simplicity.

Let $P_q$ be the diagonal operator in $l^2(\Z^n)$ defined by the formula
\begin{equation}\label{def P_q}
(P_q)_{mm}=\delta_{jm}+\delta_{j-q,m}.
\end{equation}
We define the operator $\hat{H}_q(\vec{t})$ as 
\begin{equation}\label{H_q}
\hat{H}_q(\vec{t})=H_0(\vec{t})+P_qVP_q.
\end{equation}
Note that the matrix corresponding to this operator has only two non-diagonal elements, namely,
\begin{align*}
\hat{H}_q(\vec{t})_{j,j-q}=\overline{\hat{H}_q(\vec{t})}_{j-q,j}=v_q.
\end{align*}
Thus $\hat{H}_q(\vec{t})$ has only one block \eqref{2x2}, here we assume \eqref{vq}.
The eigenvalues of this block matrix we denote by  $\hat{\lambda}^+(\vec{t})$ and $ \hat{\lambda}^-(\vec{t})$. They are given as
\begin{align}\label{block matrix}
    \hat{\lambda}^+(\vec{t})=a+b \mbox{ and }  \hat{\lambda}^-(\vec{t})=a-b,
\end{align}
where $2a=p_j^{2l}(\vec{t})+p_{j-q}^{2l}(\vec{t})$ and $2b=(4|v_q|^2+(p_j^{2l}(\vec{t})-p_{j-q}^{2l}(\vec{t}))^2)^{1/2}$. Note that $$ \hat{\lambda}^+(\vec{t})- \hat{\lambda}^-(\vec{t})=2b \geq 2|v_{q}|.$$
 This means that the spectrum of $\hat{H}_q(\vec{t})$ is $\{ \{{p_i^{2l}(\vec{t})}\}_{i\neq j, j-q}, \hat{\lambda}^{+}(\vec{t}), \hat{\lambda}^{-}(\vec{t})\}$. 
 \begin{definition} \label{d}The spectral projections of $\hat{H}_q(\vec{t})$ corresponding to  $\hat{\lambda}^{\pm }(\vec{t})$ we denote by 
 $=\hat{E}^{\pm}$.\end{definition}
Recall that $\mathcal{K}S_q(k,n-2+\delta)$ was defined in Section  1, see \eqref{map K}, \eqref{S-q}.
\begin{lemma}\label{geom lemma for singular set}
For an arbitrarily small positive $\delta$, $2\delta<2l-n$, and sufficiently large $k$, there exists a subset $\chi_q(k,\delta)$ of $\mathcal{K}S_q(k,n-2+\delta)$, such that the following conditions hold
\begin{align*}
&1. ~\mbox{there exists j}\in\mathbb{Z}^n ~\mbox{such that } \vec{P}_j(\vec{t})=\vec{k},\\
&2.~\mbox{ }|p_{j-q}^2(\vec{t})-p_{j}^2(\vec{t})|<k^{-n+2-\delta},\\
&3. \underset{i\neq j,j-q}{min}|p_{j}^2(\vec{t})-p_{i}^2(\vec{t})|>2k^{-n+2-6\delta}.
\end{align*}
for all $\vec{t}\in\chi_q(k,\delta)$.
Moreover, for any $\vec t$ in the $k$-neighborhood of $\chi_q(k,\delta)$ in $\mathbb{C}^n$, there exists a unique $j\in\mathbb{Z}^n$ such that $|p_{j}^2(\vec{t})-k^2|<5k^{-n+2-6\delta}$ and the second and third conditions above are satisfied. Also, the set $\chi_q(k,\delta)$ has an asymptotically full measure on $\mathcal{K}S_q(k,n-2+\delta)$, that is
\begin{align}
\frac{s(\mathcal{K}S_q(k,n-2+\delta)\setminus\chi_q(k,\delta))}{s(\mathcal{K}S_q(k,n-2+\delta))}=O(k^{-\delta/2}) ~\mbox{\ as } k\to\infty.
\end{align}
\end{lemma}
The previous lemma means that  $p_{j}^2(\vec{t})$ and $p_{j-q}^2(\vec{t})$ are close to each other, but they are sufficiently far away from the remaining eigenvalues.

\begin{corollary}\label{geometric cor, nonlin, 2l>n}
If $\vec{t}$ belongs to the ($2k^{-n+1-7\delta}$)- neighborhood of $\chi_{q}(k,\delta)$, then for all z on the circle $C_1^+=\{z: |z-\hat{\lambda}^+(\vec{t})|=d\}$, $d=\frac{1}{10}|v_{q}|$, both of the following inequalities are true:
\begin{align}
2|p_m^{2l}(\vec{t})-z| \geq & k^{2l-n-6\delta}, m\neq j,j-q_0,
\end{align}
\begin{align}
|\hat{\lambda}^-(\vec{t})-z| \geq d&.
\end{align}

\end{corollary}

Similar corollary holds for  ${\hat \lambda}^{-}, C_1^-$. Further for definiteness we consider ${\hat \lambda}^{+}, C_1^+$.

 Let $V$ be as in  \eqref{V} and  functions $\hat{g}^+_r(k,\vec{t})$ and $\hat{G}^+_r(k,\vec{t})$, $r\in\mathbb{N}$, $t\in\chi_q(k,\delta)$ be defined as follows:
\begin{equation} \label{tilde{g}_r}
\hat{g}^+_{r}(k,\vec{t})=\frac{(-1)^{r}}{2\pi ir}\mbox{Tr} \oint_{C_{1}^+} \left((\hat{H}_{q}(\vec{t})-z)^{-1}\hat{V}\right)^{r}dz,
\end{equation}
\begin{equation} \label{tilde{G}_r}
\hat{G}^+_{r}(k,\vec{t})=\frac{(-1)^{r+1}}{2\pi i} \oint_{C_{1}^+} \left((\hat{H}_{q}(\vec{t})-z)^{-1}\hat{V}\right)^{r}(\hat{H}_{q}(\vec{t})-z)^{-1}dz,
\end{equation}
here and below $$\hat V=V-P_qVP_q.$$
 The following result is proven in \cite{K97}.
\begin{theorem}\label{singular set main result}
Suppose $\vec{t}$ belongs to the $(k^{-n+1-7\delta})$-neighborhood in K of the set $\chi_q(k,\delta)$, $0<9\delta<2l-n$. Then for sufficiently large k, $k>k_1(V,\delta)$, in the interval $\hat{\varepsilon}(k,\delta )\equiv (\hat{\lambda}^+(\vec{t})-k^{-\delta}, \hat{\lambda}^+(\vec{t})+k^{-\delta})$, there exists a single eigenvalue of the operator H($\vec{t}$). It is given by the series
\begin{align}\label{eval in singular set}
    \hat{\lambda}^{q,+}(\vec{t})=\hat{\lambda}^+(\vec{t})+\sum _{r=2}^{\infty }\hat{g}^+_{r}(k,\vec{t}).
\end{align}
The spectral projection, corresponding to $\hat{\lambda}^q(\vec{t})$ is given by  
\begin{align}\label{proj in singular set}
    \hat{E}^{q,+}(\vec{t})=\hat{E}^+ +\sum _{r=1}^{\infty }\hat{G}^+_{r}(k,\vec{t}),
\end{align}
which converges in the class $\bf{S_1}$.
The following estimates hold for $\hat{g}^+_{r}(k,\vec{t})$, $\hat{G}^+_{r}(k,\vec{t})$:
\begin{align}
    &|\hat{g}^+_{r}(k,\vec{t})|<k^{-\gamma_{1}r-\delta},\\
    &\|\hat{G}^+_{r}(k,\vec{t})\|_1<k^{-\gamma_{1}r},
\end{align}
where 
\begin{equation} \gamma_1=(2l-n)/2-4\delta >0. \label{gamma1}\end{equation}
\end{theorem}
\begin{theorem} \label{Thm T(m)-linear} Under the conditions of Theorem \ref{singular set main result}, the series (\ref{eval in singular set}) and (\ref{proj in singular set}) can be differentiated termwise with respect to $\vec t$ any number of times, and 
they retain their asymptotic character. The coefficients $\tilde{g}_r(k,\vec t)$ and 
$\tilde{G}_r(k,\vec t)$ satisfy the following estimates in the $(k^{-n+1-7\delta })$-neighborhood in $\mathbb{C}^n$ of the singular set $\chi _q(k,\delta )$:
\begin{equation}\label{T(m) of g-linear}
\mid T(m)\hat{g}^+_r(k,\vec t)\mid <m!k^{-\gamma_1r-\delta+\mid m\mid (n-1+7\delta )} ,
\end{equation}
\begin{equation}\label{T(m) of G-linear}
\| T(m)\hat{G}^+_r(k,\vec t)\|_{1}<m!k^{-\gamma_1r+\mid m\mid (n-1+7\delta )}.
\end{equation}
\end{theorem}

\begin{corollary} \label{T(m) estimate of coeff, lin} 
The following estimates for the perturbed eigenvalue and its
spectral projection hold:
\begin{equation}
\mid T(m)(\hat{\lambda}^{q,+}(\vec t)-\hat{\lambda}^+(\vec t))\mid <2m!k^{-2\gamma_1-\delta+\mid m\mid(n-1+7\delta)}, 
\end{equation}
\begin{equation} 
\| T(m)(\hat{E}^{q,+}(\vec t)-\hat{E}^+(\vec{t}))\|_{1} <2m!k^{-\gamma_1+\mid m\mid(n-1+7\delta)}. \label{2.2.35}
\end{equation}
\end{corollary}
In particular, the following estimates are valid:
\begin{align}
    |\hat{\lambda}^{q,+}(\vec{t})-\hat{\lambda}^+(\vec{t})|<2k^{-2\gamma_1-\delta},
\end{align}
\begin{align}
    \| \hat{E}^{q,+}(\vec{t})-\hat{E}^+\|_1<  2k^{-\gamma_1},
\end{align}
\begin{align}
    |\nabla \hat{\lambda}^{q,+}(\vec{t})-\nabla \hat{\lambda}^+(\vec{t})|<2k^{-2\gamma_1+n-1+6\delta}.
\end{align}
\begin{corollary}\label{12.26a} An eigenfunction $u_0^+(\vec x, \vec t)$ corresponding to $ \hat{E}^{q,+}(\vec{t})$ satisfies the estimate:
\begin{equation}
\label{12.26b}
\left\|e^{-i<\vec k, \vec x>}u_0^+(\vec x, \vec t)-\big(a_1+a_2e^{- i<\vec{P}_q(0), \vec x>}\big)\right \|_*<ck^{-\gamma_1}, \end{equation}
where $(a_1, a_2)$ is an eigenvector of \eqref{2x2} corresponding to $ \hat{\lambda}^+(\vec{t})$.
\end{corollary}
The analogous results hold for $\hat{\lambda}^{q,-}(\vec{t}), \hat{E}^{q,-}(\vec{t})$.
\begin{corollary} \label{17}The surface $\hat{\lambda}^{q,\pm } (\vec{t})=\lambda _0 $ is in the  real $\lambda _0^{-2\gamma _1-2l+1-\delta}$-neighborhood of  $ \chi _q (k,\delta )$ for every sufficiently large $\lambda _0$. \end{corollary}

The results above hold for $V: \|V\|_*<\infty $.
\subsection{Nonlinear Periodic Polyharmonic Equation}
%
In this section we consider equation (\ref{main equation}) with the  quasi-periodic boundary conditions \eqref{main condition}.
In \cite{KKS1,KKS2} we  proved
 existence of a quasi-periodic solution of equation (\ref{main equation}) being close to  a plane wave $Ae^{i\langle{\vec{k},\vec{x}}\rangle}$ for every $\vec{k}$   belonging to a non-resonant set $\chi _0 (k)$, where $A$ is a complex number with sufficiently small $|A|$.
The main idea of the papers is to look at equation (\ref{main equation}) as a linear equation. This can be done by considering the sum $V(\vec{x})+\sigma|u(\vec{x})|^2$ as an unknown potential. 
We use the facts proved for the linear case.

Now we describe  the technique used to find a function $u$ that solves the nonlinear equation. Let $ \hat{{\l}}_1$ be the space of functions in $Q$ with Fourier coefficients in ${\l}_1(\Z^n)$, see  \eqref{norm ||_*}. First, a sequence of operators $\{W_m\}_{m=0}^{\infty}$ is constructed via the map $\mathcal{M}: \hat{{\l}}_1 \to  \hat{{\l}}_1$ defined as 
\begin{equation}\label{Def of map M}
\mathcal{M}W(\vec{x})=V(\vec{x})+\sigma |u_{\widetilde{W}}(\vec{x})|^2.
\end{equation}
Here, $u_{\widetilde{W}}$ is an eigenfunction of the linear operator $(-\Delta)^{l}+\widetilde{W}$ with the same boundary conditions (\ref{main condition}), where $\widetilde{W}$ is defined by the formula
\begin{equation}\label{Def of W}
\widetilde{W}(\vec{x})=W(\vec{x})-\frac{1}{(2\pi)^n}\int_{Q} W(\vec{x}) \,d\vec{x}.
\end{equation}
More precisely, $u_{\widetilde{W}}$ is defined as
\begin{equation}
\begin{split} \label{10}
u_{\widetilde{W}}(\vec{x}) & =\sum_{s\in\mathbb{Z}^n}{(E_{\widetilde{W}}(\vec{t}))_{sj}e^{i\langle{\vec{P}_s(\vec{t}),\vec{x}}\rangle}}\\
& = e^{i\langle{\vec{P}_j(\vec{t}),\vec{x}}\rangle}+\sum_{r=1}^{\infty}\sum_{s\in\mathbb{Z}^n}(G_{\widetilde{W}}^r(k,\vec{t}))_{sj} e^{i\langle{\vec{P}_s(\vec{t}),\vec{x}}\rangle},
\end{split}
\end{equation}
here $j: \vec P_j(\vec t)=\vec k $. Formula \eqref{10} is analogous to \eqref{proj op 2l>n}.
Second, the operator $W_m$ is defined by a recurrence procedure as follows. Let $W_0 = V + \sigma |A|^2$. Define
\begin{equation} \label{13}
W_{m+1}=\mathcal{M}W_m, ~m=0,1,2,...
\end{equation} 
Then,  $W_m$ is proved to be a Cauchy sequence of periodic potentials converging to a periodic potential $W$ with respect to the norm $\|\cdot \|_*$, see \eqref{norm ||_*}.
 Note that for any $m$, there is a solution $u_m$ of the equation 
\begin{equation*}
(-\Delta)^lu_m+\widetilde{W}_mu_m=\lambda_mu_m,
\end{equation*}
given by the formula  \eqref{10}.
Moreover, the sequences of the functions $\{u_m\}_{m=0}^{\infty}$ and the eigenvalues $\{\lambda_m\}_{m=0}^{\infty}$ converge to a function $u_{\widetilde{W}}$
and a real number $\lambda$, respectively. It is shown that $u=u_{\widetilde{W}}$ solves the nonlinear equation (\ref{main equation}).
The limits $E_{\widetilde W}(\vec{t})$ and $\lambda_{\widetilde W}(\vec{t})$ of the sequences $\{E_m(\vec{t})\}_{m=0}^{\infty}$ and $\{\lambda_m(\vec{t})\}_{m=0}^{\infty}$, respectively, are given by the two  series \eqref{ev 2l>n}, \eqref{proj op 2l>n}
written for the potential $\widetilde{W}$ instead of $V$.
 The following theorem holds when $2l>n$.
 \begin{theorem} \label{Kim} Let $0<2\delta <2l-n$. Suppose $\vec{t}$ belongs to the $(k^{-n+1-2\delta})$-neighborhood in $K$ of the non-resonant set $\chi_0(k,\delta )$,  $k > k_{1}(\|V\|_*,\delta)$, and $A\in \mathbb{C}: |\sigma||A|^2 <k^{\gamma_0-\delta}$, where 
 \begin{equation} \gamma_0=2l-n-2\delta . \label{gamma0}
 \end{equation} Then,
there is  a function $u$, depending on $\vec t $ as a parameter, and a real value $\lambda(\vec t)$, such that they
solve the  equation \eqref{main equation}
and  the quasi-periodic boundary condition (\ref{main condition}). The following estimates hold:
\begin{align}\label{eval of sol, 4l>n+1}
\lambda(\vec{t})=&p_j^{2l}(\vec t)+\sigma|A|^2+O\left(k^{-\gamma _0} \right),
\end{align} 
\begin{align} \label{sol u, 4l>n+1}
u( \vec{x})=&Ae^{i\langle{ \vec{P}_{j}(\vec t), \vec{x} }\rangle}\left(1+\tilde{u}( \vec{x})\right),
\end{align}
where
 $\tilde{u}$ is periodic and \begin{equation}
\|\tilde{u}\|_{*}< k^{-\gamma _0}.
\end{equation}
Moreover,
the following estimates hold:
\begin{equation}
\left| \nabla \left(\lambda (\vec t)-p_j^{2l}(\vec t)-\sigma |A|^2\right)\right| <k^{2l-1-\gamma_0}.
\end{equation}
%
\end{theorem}
\begin{corollary} The surface $\lambda (\vec{t})=\lambda _0$, $\lambda _0\equiv k_0^{2l}+\sigma |A|^2$, is in the $k_0^{-(4l-n+1-2\delta )}$-neighborhood of  $ \chi _0 (k,\delta )$ for every fixed $A: |\sigma A^2|<k_0 ^{\gamma _0-\delta }$ and sufficiently large $\lambda _0$.
\end{corollary}
 If $l=1, n=2$ an analogous theorem holds with somewhat different non-resonant set set $\chi_0(k,\delta )$, and a constant $\gamma _0$. 


%
%
%

\section{Perturbation Theory for Non-linear Polyharmonic  Equation with Periodic Potential for $2l>n$. Resonant Case.}

  In  Section 3.1, we introduce  maps $\mathcal{\hat{M}}^{\pm}$ and construct two sequences of potentials $\{W_m^{\pm}\}_{m=0}^\infty$. We will prove that they are Cauchy sequences converging to some potentials $W^{\pm}$.
  In Section 3.2 we prove the existence of  solutions $u^{\pm}(\vec t, \vec{x})$, $\lambda=\lambda^{\pm} (\vec t)$  of  (\ref{main equation}), (\ref{main condition})  and obtain their estimates at the high energy region. The functions $u^{\pm}(\vec t, \vec{x})$ are shown to be close to  non-trivial combinations of two plane waves. Finally, in Section 3.3, we derive estimates for the  derivatives of $\lambda ^{\pm}(\vec t)$ and $u^{\pm}(\vec t,\vec x)$ with respect to $\vec t $. The isoenergetic surface  $\{t\in K, \lambda ^{\pm}(\vec t)=\lambda _0\} $ is constructed for the resonant case for every sufficiently large $\lambda _0$.
 \subsection{Construction of a Cauchy Sequence}
 Recall that the operator $\hat{H}_{q}(\vec{t})$ defined by (\ref{H_q}) has the spectrum $\{\{{p_i^{2l}(\vec{t})}\}_{i\neq j,j-{q}}, \hat{\lambda}^{+}(\vec{t}), \hat{\lambda}^{-}(\vec{t})\}$. It is easy to verify that  eigenvectors of the block submatrix
of $\hat{H}_{q}(\vec{t})$ corresponding to the eigenvalues $\hat{\lambda}^{\pm}(\vec{t})$, see (\ref{block matrix}), are:
\begin{align}\label{eigenvector of lambda +}
    \hat{e}^{+} = (p_{j-{q}}^{2l}(\vec{t}) - \hat{\lambda}^{+}(\vec{t}),\ -v_{q})^T
\end{align}
\begin{align}\label{eigenvector of lambda -}
    \hat{e}^- = (-v_{-{q}}, \ p_{j}^{2l}(\vec{t})-\hat{\lambda}^-(\vec{t}))^T.
\end{align}
By Definition 1 the corresponding spectral projections are denoted by by $\hat E^{\pm}(\vec t)$. 
It is clear that $|\hat{\lambda}^+(\vec{t})-p_j^{2l}(\vec{t})|<2|v_{q}|$ and $|\hat{\lambda}^-(\vec{t})-p_j^{2l}(\vec{t})|<2|v_{q}|$. 
%
 We assume that $|v_{q}| \neq 0$. 
 
 The cases $+$ and $-$ are analogous. Further, for definiteness we consider the case $+$.
 
 Wee use a geometric Lemma \ref{geom lemma for singular set} and Corollary  \ref{geometric cor, nonlin, 2l>n}. 
\begin{definition}\label{def of u}  Let $W$ be such that $\|W\|_*<\infty $, 
 $\vec t\in \chi _q(k,\delta )$ and 
\begin{align}\label{def of u eq-e}
u^+_{{W}}(\vec{x})= \psi ^+_{{W}}(\vec{x})e^{i\langle{\vec{P}_j(\vec{t}), \vec{x} }\rangle},
\end{align}
where $\psi ^+_{{W}}$ is periodic and
\begin{align}\label{def of psi}
 \psi ^+_{{W}}(\vec{x})=A\sum_{b\in\mathbb{Z}^n}{\hat E_{{W}}^{q,+}(\vec{t})_{j-b,j}e^{i\langle{\vec{P}_{b}(0),\vec{x}}\rangle}}
\end{align}
$\hat E_W^{q,+}$ being given by \eqref{tilde{G}_r},  \eqref{proj in singular set} with $W-P_qVP_q$  in  \eqref{tilde{G}_r} instead of $\hat V$.
\end{definition}
Obviously, $ u^+_W(x)$ satisfies $(-\Delta )^lu_W^++W u_W^+= \lambda ^+_{W}(t)u_W^+$,
where $\lambda ^+_W (t)$ is given by \eqref{eval in singular set} with $W-P_qVP_q$  in  \eqref{tilde{G}_r} instead of $\hat V$.
\begin{definition}
Let the map $\hat{\mathcal{M}}^+: \hat{\l}_1\rightarrow \hat{\l}_1$ be defined by
\begin{align}\label{map M}
    \hat{\mathcal{M}}^+W(\vec{x})=V(\vec{x})+\sigma |  u^+_{W}(\vec{x})|^2,
\end{align}
here  $ \hat{{\l}}_1$ is the space of  functions in $Q$ with Fourier coefficients in ${\l}_1(\Z^n)$, see  \eqref{norm ||_*}.
\end{definition}
%
Next, we define the sequence $\{W_m^+\}_{m=0}^{\infty}$ using the map $\mathcal{\hat{M}}^+$. Let 
\begin{align}\label{W_0}
    W_0^+=V, 
\end{align}
\begin{align}\label{W_m}
    W_m^+=\hat{\mathcal{M}}^+W^+_{m-1}.
\end{align}
 For every $m$ there is an eigenfunction $ u^+_{m}(\vec{x})$ corresponding ${W}_{m}^+$, see \eqref{def of u eq-e}, and the corresponding eigenvalue number ${\lambda}_{m}^+$. Note that this sequence $W_m^+$ is quite different from that for a non-resonance case given by \eqref{13}.

\begin{definition}\label{def3} Let, $\hat E_{-1}^+\equiv \hat E^+$ be the spectral projection of the model operator $\hat H_{q}$ corresponding to $\lambda^+(\vec t)$, and $\hat E^+_{0}=\hat E^{q,+}$, see 
\eqref{proj in singular set}.  By analogy, 
 $\hat E_m^+\equiv \hat E^{q,+}_{{W_m^+}}(\vec{t})$,  the formula 
\eqref{tilde{G}_r}, \eqref{proj in singular set} being used  for $ E^+_{{W_m^+}}(\vec{t})$ with $W_m^+-P_qVP_q$ instead of $\hat V$ . \end{definition}
Clearly, $\hat E_{-1}^+$ is a one-dimensional projection corresponding to \eqref{eigenvector of lambda +}.
\begin{definition} \label{def4} Let
\begin{align}\label{def of u_-1}
u^+_{-1}(\vec{x})=\psi ^+_{-1}(\vec{x})e^{i\langle{\vec{P}_j(\vec{t}), \vec{x} }\rangle},
\end{align}
\begin{align}\label{def of u_0}
u_{0}^+(\vec{x})=\psi ^+_{0}(\vec{x})e^{i\langle{\vec{P}_j(\vec{t}), \vec{x} }\rangle},
\end{align}
where $\psi_{-1}^+$ and $\psi_0^+$ are the periodic parts of $u_{-1}^+$ and $u_0^+$,   written
in terms of $\hat{E}^+(\vec{t})$ and $E^{q,+}(\vec{t})$ as follows. First,
\begin{align}\label{def of psi_-1}
\psi_{-1}^+(\vec{x})=A\Big(\hat{E}^+(\vec{t})_{j,j}+\hat{E}^+(\vec{t})_{j-q,j}e^{-i\langle \vec{P}_q(0), \vec{x} \rangle }\Big)
\end{align}
The sum in \eqref{def of psi_-1} has only two terms, because $\hat{E}^+_{lm}=0$ if $l,m\neq 0,-q$. Note that $u^+_{-1}(\vec{x})$ is a linear combination of $e^{i\langle \vec{P}_j(\vec{t}), \vec x\rangle}$  and $e^{i\langle \vec{P}_{j-q}(\vec{t}), \vec x \rangle}$.

Second, 
\begin{align}\label{def of psi_0}
\psi_{0}^+(\vec{x})=A\sum_{b\in\mathbb{Z}^n}{\hat E^{q,+}(\vec{t})_{j-b,j}}e^{i\langle{\vec{P}_b(0), \vec{x} }\rangle},
\end{align}
see \eqref{proj in singular set} and Definition \ref{def3}.  

Next, by analogy to (\ref{def of u_-1}) and (\ref{def of psi_-1}), 
\begin{align}\label{def of u_s}
u_{s}^+(\vec{x})=\psi_{s}^+(\vec{x})e^{i\langle{\vec{P}_j(\vec{t}), \vec{x} }\rangle},
\end{align}
 $\psi_s^+$ being periodic and given by the formula:
\begin{align}\label{def of psi_s}
\psi_{s}^+(\vec{x})=A\sum_{b\in\mathbb{Z}^n}{\hat{E}_s^+(\vec{t})_{j-b,j}}e^{i\langle{\vec{P}_b(0), \vec{x} }\rangle}.
\end{align}
\end{definition}

\begin{definition}
Let $T$ be an operator  in $l_2(\mathbb{Z}^n)$. Then $\|T\|_0$ is defined by
\begin{align}\label{def of norm ||_o}
\|T\|_0=\frac{1}{2}\underset{r}{max}\sum_{p}\left(|T_{pr}|+|T_{rp}|\right).
\end{align}
It is obvious that 
$ \|\mathcal{A}\mathcal{B}\|_0\leq 2\|\mathcal{A}\|_0\|\mathcal{B}\|_0$.
\end{definition}




Remind that $2\gamma _1=2l-n-8\delta $, see \eqref{gamma1}.
\begin{lemma}\label{estimates}
Let $\gamma _2>0$, $0<8\delta <2l-n$.
 The following estimates hold for every m=1, 2, ..., and $\forall A\in \mathbb{C}^{}: |\sigma||A|^{2}<k^{{-\gamma}_2}$ when  $\vec t\in \chi _q(k,\delta )$:
\begin{align}\label{estimate of W_m difference}
    \|W_m^+-W^+_{m-1}\|_* < (\hat{c}k^{-{\gamma} })^m, 
\end{align}
\begin{align}\label{estimate of V-W_m+++}
     \|W^+_{m}-V\|_* <\sum_{r=1}^{m}(\hat{c}k^{-{\gamma} })^{r},
\end{align}
\begin{align}\label{estimate of E^_0 difference+}
\|\hat{E}_{0}^+(\vec{t})-\hat{E}^+_{-1}(\vec{t})\|_0<\hat{c}k^{-{\gamma} },
\end{align}
\begin{align}\label{estimate of E^_m difference}
\|\hat{E}_{m-1}^+(\vec{t})-\hat{E}_{m-2}^+(\vec{t})\|_0<\hat c (\hat{c}k^{-{\gamma}})^{m-1}, \ m \geq 2, \ \ 
\end{align}
where \begin{equation} \gamma =\min \{\gamma _1, \gamma _2\},\label{gamma} \end{equation} $\hat{c}=\hat{c}{(V)}$  and $k$ is sufficiently large: $k>k_0(V, \gamma _2, \delta)$. \end{lemma}
%
\begin{corollary}\label{Cauchy seq cor}
The sequence $\{W_m^+\}_{m=0}^{\infty}$ converges to a continuous and periodic potential $W^+$ with respect to the norm $\|\cdot\|_{*}$. The following estimate holds:
\begin{align}\label{diff W-W_m}
\|W^+-W_m^+\|_*<2(\hat{c}k^{-{\gamma} })^{m+1}.
\end{align}
\end{corollary}
\begin{proof} 
We use an induction. For the first step  $m=1$, we need to show \eqref{estimate of E^_0 difference+} and 
\begin{equation}\label{base case of induct 2}
\|W_1^+-W_0^+\|_*\leq\hat{c}k^{-{\gamma }}.
\end{equation}
We start with \eqref{estimate of E^_0 difference+}. This is a perturbative formula for a linear operator.  It is given by  \eqref{2.2.35} for $m=0$ up to the notations  in Definition \ref{def3}.
 To prove \eqref{base case of induct 2} we consider  two functions $u^+_{-1}(\vec{x})$ and $u^+_0(\vec{x})$, see Definition \ref{def4}.
Next, by \eqref{map M} -- \eqref{W_m},
\begin{align}\label{W_1 - W_0}
\|W^+_{1}-W_{0}^+\|_{*} = &\big\||\sigma||u^+_{0}|^{2}\big\|_{*}\\
\nonumber
 \leq & |\sigma|\big(\||u^+_{0}|^{2}-|u^+_{-1}|^{2}\|_{*}+\|u^+_{-1}\|_{*}^2\big)\\
\nonumber= &|\sigma|\big(\||\psi ^+_{0}|^{2}-|\psi ^+_{-1}|^{2}\|_*+\|\psi ^+_{-1}\|_{*}^2\big)\\
\nonumber\leq& |\sigma|\big(\||\psi ^+_{0}|^{2}-|\psi ^+_{-1}|^2+2i\Im(\bar \psi ^+_{-1}\psi ^+_{0})\|_{*}+\|\psi ^+_{-1}\|_{*}^2\big)\\
\nonumber=&|\sigma|\big(\|(\psi_{0}^+-\psi ^+_{-1})(\bar{\psi} ^+_{0}+\bar \psi ^+_{-1})\|_{*}+\|\psi ^+_{-1}\|_{*}^2\big)\\
\nonumber \leq& |\sigma|\big(\|\psi ^+_{0}-\psi ^+_{-1}\|_{*}\|\bar{\psi} ^+_{0}+\bar \psi ^+_{-1}\|_*+\|\psi ^+_{-1}\|_{*}^2\big).
\end{align}
By \eqref{2.2.35} for $m=0$,
\begin{align}
\|\psi ^+_{0}-\psi^+_{-1}\|_*\leq &|A|\|\hat E^{+q}(\vec t)-\hat E^+(\vec t)\|_0\leq 2|A|k^{-{\gamma_1}},
\label{estimate of diff of psi}
\end{align}
Next, we notice that  $\|\bar{\psi}_{-1}^+\|_*=\|\psi ^+_{-1}\|_*$ and $|\hat{E}_{jj}^+|\leq 1$ and $|\hat{E}_{j-q,j}^+|\leq 1$. This implies that $\|\psi ^+_{-1}\|_*\leq 2|A|$ because the only two nonzero elements of $\hat{E}^+(\vec{t})$  are $\hat{E}_{jj}^+$ and $\hat{E}_{j-q,j}^+$ . Hence, by (\ref{estimate of diff of psi}),
\begin{align*}
\|\bar{\psi}^+_{0}+\bar \psi ^+_{-1}\|_*\leq &\|\bar{\psi}^+_{0}-\bar \psi ^+_{-1}\|_*+2\|\bar \psi ^+_{-1}\|_*
\leq 2  |A|k^{-{\gamma_1}}+4|A|.
\end{align*}
Since $k$ is chosen to be sufficiently large, we can write
\begin{align}\label{estimate of sum of psi}
\|\bar{\psi} ^+_{0}+\bar \psi ^+_{-1}\|_*\leq 5|A|.
\end{align}
 Considering the last line of (\ref{W_1 - W_0}) and using  (\ref{estimate of diff of psi}), (\ref{estimate of sum of psi}) and  the condition $|\sigma A|^2<k^{-{\gamma}_2},$ we obtain that
\begin{align}
\||W_{1}^+-W_{0}^+\|_* \leq & |\sigma|\big(\big(2|A|k^{-{\gamma_1}}\big)5|A|+2|A|^2\big)
 \leq & 10|\sigma||A|^2 k^{-{\gamma_1}}+2|\sigma||A|^2\\
\leq  3 k^{-2l{\gamma _2}}.\ \  \label{base case proof}
\end{align}
for sufficiently large $k>k_0(\|V\|_*, d, \gamma _1)$.
This proves (\ref{estimate of W_m difference}) and (\ref{estimate of V-W_m+++}) for $m=1$.

Now, let us assume that (\ref{estimate of W_m difference}) -\eqref{estimate of E^_m difference}  are satisfied for all $s=1,2,...,m-1$, $m\geq 2$, i.e.,
\begin{align}\label{estimate of W_s difference}
    \|W_s^+-W_{s-1}^+\|_* \leq (\hat{c}k^{-\gamma })^s, \end{align}
\begin{align}\label{estimate of V-W_s}
     \|W^+_{s}-V\|_* \leq \sum_{r=1}^{s}(\hat{c}k^{-{\gamma}})^{r},
\end{align}
\begin{align}\label{estimate of E^_s difference}
\|\hat{E}_{s-1}^+(\vec{t})-\hat{E}^+_{s-2}(\vec{t})\|_0\leq \hat c (\hat{c}k^{-\gamma})^{s-1}.
\end{align}
First, let us prove (\ref{estimate of E^_m difference}). Considering that $H_0+W_{m-1}^+=\hat H_q+(W_{m-1}^+-P_qVP_q)$
we define:
\begin{equation}\label{B^_s def}
\hat{B}^+_{m-1}(z)=(\hat{H}_q(\vec{t})-z)^{-\frac{1}{2}}\left(W^+_{m-1}-P_qVP_q\right)(\hat{H}_q(\vec{t})-z)^{-\frac{1}{2}}.
\end{equation}
It is quite obvious that 
\begin{align} 
\label{align} \|(I-P_q)(\hat{H}_q(\vec{t})-z)^{-\frac{1}{2}}\|_0=\|(\hat{H}_q(\vec{t})-z)^{-\frac{1}{2}}(I-P_q)\|_0\leq k^{-{\gamma_{1}}};\\
\label{align+} \|P_q(\hat{H}_q(\vec{t})-z)^{-\frac{1}{2}}\|_0=\|(\hat{H}_q(\vec{t})-z)^{-\frac{1}{2}}P_q\|_0\leq \frac{1}{\sqrt{d}};
\end{align}
 here $z \in C_1^+$, $d$  is the radius of $C_1^+$, $d=\frac{1}{10}|v_{q}|$.
Considering two above estimates and  \begin{equation} \label{new}
W^+_{m-1}-P_qVP_q=\hat W^+_{m-1} +P_q(W^+_{m-1}-V)P_q,,\ \mbox{  here }  \hat W^+_m=W^+_m-P_qW^+_mP_q, \end{equation} we obtain:
\begin{align} \label{21.12}
\underset{z\in C_1}{max}{\|\hat{B}^+_{m-1}(z)}\|_0\leq (1+\frac{2}{\sqrt{d}})\|\hat{W}^+_{m-1}\|_0k^{-{\gamma_1}}
+d^{-1}\|P_q(W^+_{m-1}-V)P_q\|_0.
\end{align}
Hence, by induction assumption \eqref{estimate of V-W_s},
\begin{align}\label{estimate of B^_s}
\underset{z\in C_1}{max}{\|\hat{B}^+_{m-1}(z)}\|_0\leq  \beta k^{-{\gamma }}. 
\end{align}
where $ \beta (V) = 3(1+\frac{2}{\sqrt{d}})\|V\|+2\hat c d^{-1}$.
This means that 
\begin{align}\label{estimate of G^_{m-1},r}
\|\hat{G}^+_{{m-1},r}(k,\vec{t})\|_0\leq \big(c \beta k^{-\gamma}\big)^r,
\end{align}
where $\hat{G}^+_{{m-1},r}(k,\vec{t})$ is given by
\eqref{tilde{G}_r} with ${W}^+_{m-1}-P_qVP_q$ instead of $\hat{V}$ and $c$ is an absolute constant.
Clearly, 
\begin{align*}
    \|\hat{E}^+_{m-1}(\vec{t})\|_0\leq|\hat{E}^+(\vec{t})_{jj}|+|\hat{E}^+(\vec{t})_{j-q,j}|+\sum_{r=1}^{\infty}\|\hat{G}^+_{{m-1},r}(k,\vec{t})\|_0.
\end{align*}
Since $|\hat{E}^+(\vec{t})_{jj}|, |\hat{E}^+_{m-1}(\vec{t})|\leq 1$ and by (\ref{estimate of G^_{m-1},r}), we get $\|\hat{E}^+_{m-1}(\vec{t})\|_0\leq 3$.

Now, we note that for $m\geq 2$:
\begin{align}
\nonumber\|\hat{E}^+_{{m-1}}(\vec{t})-\hat{E}^+_{m-2}(\vec{t})\|_0\leq \sum_{r=1}^{\infty}\big\|\hat{G}^+_{{m-1},r}(k,\vec{t})-\hat{G}^+_{m-2,r}(k,\vec{t})\big\|_0\\
\nonumber \leq \frac{1}{2\pi} \sum_{r=1}^{\infty} \big\|\oint_{C_{1}^+} \big(\hat{H}_{q}(\vec{t})-z\big)^{-\frac{1}{2}}\big[\hat{B}^+_{m-1}(z)^{r}-\hat{B}^+_{m-2}(z)^{r}\big]\big(\hat{H}_{q}(\vec{t})-z\big)^{-\frac{1}{2}}dz \big\|_0,\\
\nonumber \leq \left( \frac{1}{2\pi} \int_{C_{1}^+}\big\| \big(\hat{H}_{q}(\vec{t})-z\big)^{-1/2}\big\|_0^2 ds \right)\left(\sum_{r=1}^{\infty}
 \underset{z\in C_1^+}{\max}\big\|\hat{B}^+_{m-1}(z)^r-\hat{B}^+_{m-2}(z)^r\big\|_0\right)\\
\leq  c\sum_{r=1}^{\infty}\underset{z\in C_1^+}{max}\big\|(\hat{B}^+_{m-1}(z)^r-\hat{B}^+_{m-2}(z)^r\big\|_0. \label{diff of B_s}
\end{align}
Consider $\|\hat{B}^+_{m-1}(z)-\hat{B}^+_{m-2}(z)\|_0$ for any $z\in C_1^+$. Arguing as in the proof of the estimate  for $\|\hat{B}^+_{m-1}(z)\|_0$, see \eqref{21.12}, (\ref{estimate of B^_s}), 
we obtain:
\begin{align}\label{estimate diff of B_s}
\|\hat{B}^+_{m-1}(z)-\hat{B}^+_{m-2}(z)\|_0
\leq &\beta\|W^+_{m-1}-W^+_{m-2}\|_0, \ \ m\geq 2.
\end{align}
On the other hand, the right-hand side of (\ref{diff of B_s}), which we  denote here by R for simplicity, can be estimated by using (\ref{estimate of W_s difference}), (\ref{estimate of B^_s}), and (\ref{estimate diff of B_s}):
\begin{align}
\nonumber R\leq &c\sum_{r=1}^{\infty}\underset{z\in C_1^+}{max}\big[\|\hat{B}^+_{m-1}(z)-\hat{B}^+_{m-2}(z)\|_0\big(\|\hat{B}^+_{m-1}(z)\|_0+\|\hat{B}^+_{m-2}(z)\|_0\big)^{r-1}\big]\\
\nonumber \leq &c\sum_{r=1}^{\infty}\beta\|W^+_{m-1}-W^+_{m-2}\|_0\big(2\beta k^{-{\gamma }}\big)^{r-1}
%
%
 \leq c\beta\|W^+_{m-1}-W^+_{m-2}\|_0 \\
\leq &c \beta (\hat{c}k^{-\gamma })^{m-1}.\label{R}
\end{align}
Thus, \eqref{estimate of E^_m difference} holds.
Next, we consider $\|W^+_m-W^+_{m-1}\|_0$.  
Obviously, 
\begin{align}\label{norm of psi_s}
\|\psi_s^+\|_*\leq|A|\|\hat{E}^+_s(\vec{t})\|_0.
\end{align}
Using \eqref{map M}, we easily get  
\begin{align}\label{W_m - W_{m-1}}
\|W^+_m-W^+_{m-1}\|_0=\left\|\sigma |u^+_{m-1}|^2-\sigma |u^+_{m-2}|^2\right\|_* \leq |\sigma|\|\psi^+_{m-1}-\psi^+_{m-2}\|_{*}\|\bar{\psi}^+_{m-1}+\bar \psi ^+_{m-2}\|_*.
\end{align}
Now, formula (\ref{W_m - W_{m-1}}) can be rewritten using  \eqref{def of psi_s} and \eqref{norm of psi_s}:
\begin{align}
\nonumber \|W^+_m-W^+_{m-1}\|_0\leq &|\sigma||A|^2\|\hat{E}^+_{m-1}(\vec{t})-\hat{E}^+_{m-2}(\vec{t})\|_0\big(\|\hat{E}^+_{m-1}(\vec{t})\|_0+\|\hat{E}^+_{m-2}(\vec{t})\|_0\big)\\
\leq &(\hat{c}k^{-{\gamma}})^{m},\ \  m\geq 2. \label{induct step}
\end{align}
Therefore, (\ref{estimate of W_m difference}) and  (\ref{estimate of V-W_m+++}) hold.

\end{proof}
\subsection{Solution to Nonlinear Polyharmonic Equation with Periodic Potential for $2l>n$}
In this section we show that  convergence of $\{W^+_m\}_{m=0}^{\infty}$ leads to  convergence of the sequence of the spectral projections $\{\hat{E}_m^+(\vec{t})\}_{m=0}^{\infty}$ to that  of the operator ${H}_0(\vec{t})+{W^+}$ (in the norm $\|\cdot\|_0$).
 The sequence of the corresponding eigenvalues $\{\hat{\lambda}^+_m(\vec{t})\}_{m=0}^{\infty}$ converges to the corresponding eigenvalue of  ${H}_0(\vec{t})+{W}^+$.
\begin{lemma}\label{conv of E^_m} Let $\gamma _2>0$, $0<9\delta <2l-n$.
Suppose that $\vec t$ belong to the $(k^{-n+1-7\delta})$-neighborhood in $K$ of the set $\chi _{q}(k,\delta)$. Then, for every sufficiently large $k>k_{0}(V,\gamma _2, \delta)$ and every $A\in \mathbb{C}:|\sigma||A|^{2}<k^{-{\gamma}_2}$, the sequence $\{\hat{E}^+_m(\vec{t})\}_{m=0}^{\infty}$ converges with respect to the norm $\|\cdot\|_0$ to a one-dimensional spectral projection $\hat E^+_{{W}}(\vec{t})$ of ${H}_0(\vec{t})+{W}^+$,
\begin{align}\label{estimate E^_m -E}
\|\hat{E}^+_m(\vec{t})-\hat E^+_{{W}}(\vec t)\|_0\leq  \hat c(\hat{c}k^{-{\gamma}})^{m+1},  \ m=0,1,...
\end{align} 
The projection $\hat E^+_{{W}}(\vec{t})$ is given as
\begin{align}\label{def of E_W^}
\hat E^+_{{W}}(\vec{t})= \hat{E}^{+}(\vec{t}) + \sum_{r=1}^{\infty}\hat G^+_{{W},r}(k,\vec{t}),
\end{align}
where. $\hat G^+_{{W},r}(k,\vec{t})$ is given by \eqref{tilde{G}_r} with  with ${W}^+-P_qVP_q$ instead of $\hat{V}$.
The following estimate holds:
\begin{align}\label{G_W^ estimate}
\|\hat G^+_{{W},r}(k,\vec{t})\|_0\leq \hat c \left(\hat{c}k^{-{\gamma}}\right)^{r}.
\end{align}
\end{lemma}
\begin{proof}
Let  $B^+_{{W}}(z)$ be defined by (\ref{B^_s def}) with ${W}^+$ instead of $W_{m-1}^+$.
Next, we estimate
\begin{align}
\nonumber \|\hat{B}_m^+(z)-\hat B^+_{{W}}(z)\|_0=\|\big(\hat{H}_q(\vec{t})-z\big)^{-\frac{1}{2}}({W}^+_m-{W}^+)\big(\hat{H}_q(\vec{t})-z\big)^{-\frac{1}{2}}\|_0\end{align}
By \eqref{align}, \eqref{align+}
\begin{align}
 \|\hat{B}^+_m(z)-\hat B^+_{{W}}(z)\|_0 \leq 
\ & \beta \|W_m^+-W^+\|_*. \label{15}
\end{align}
Using Corollary \ref{Cauchy seq cor}, we get
\begin{align}
 \|\hat{B}^+_m(z)-\hat B^+_{{W}}(z)\|_0 \leq &
\hat c (\hat{c}k^{-{\gamma}})^{m+1}. \label{B^_m -B_W}
\end{align}
%
By \eqref{estimate of B^_s},
\begin{align}
\underset{z \in C_1}{max}\|\hat B^+_{W}(z)\|_0 &
\leq  \beta k^{-{\gamma}}. \label{estimate of B^(z)}
\end{align}
Next, we prove (\ref{estimate E^_m -E}). Let us write
\begin{equation}\label{16}
\left\|\hat{E}^+_m(\vec{t})-\hat E^+_{W}(\vec t)\right\|_0 \leq \sum_{r=1}^{\infty}\big\|\hat{G}^+_{m,{r}}(k,\vec{t}) -\hat G^+_{{W},r}(k,\vec{t})\big\|_0\end{equation}
$$\leq c\sum_{r=1}^{\infty} \underset{z \in C_1^+}{max}\big\|\hat{B}^+_m(z)^r-\hat B^+_{W}(z)^r\big\|_0 $$
$$ \leq c\sum_{r=1}^{\infty}\underset{z \in C_1^+}{max}\Big[\|\hat{B}^+_m(z)-B^+_{{W}}(z)\|_0\Big(\|\hat{B}^+_m(z)\|_0+\|B^+_{{W}}(z)\|_0\Big)^{r-1}\big]. $$

By (\ref{estimate of B^_s}) and  (\ref{15}),
\begin{align}
\nonumber \|\hat{E}^+_m(\vec{t})-\hat E^+_{W}(\vec{t})\|_0 \leq & \sum_{r=1}^{\infty}\big[
d^{-1}\|W_m^+-W^+\|_*\big(2\beta k^{-{\gamma}}\big)^{r-1}\big]\\
\nonumber \leq & 
\hat c\|W_m^+-W^+\|_*.
\end{align}
Hence, by (\ref{diff W-W_m})
we get (\ref{estimate E^_m -E}).
Since $\|W_m^+-W^+\|_* \rightarrow 0$ as $m \rightarrow \infty$, we have that $\hat E^+_{{W}}(\vec{t})$ is the limit of $\hat{E}^+_m(\vec{t})$. Formulas above also prove that $\hat{G}^+_{m,{r}}(k,\vec{t})$ converges to $\hat G^+_{{W},r}(k,\vec{t})$ in the norm $\|\cdot\|_0$.
Now, it remains to prove (\ref{G_W^ estimate}).  Indeed,
\begin{align}
\nonumber \|\hat G^+_{{W},{r}}\|_0 = &\Big\|\frac{(-1)^{r+1}}{2\pi i}\oint_{C_{1}^+} \big(\hat{H}_{q}(\vec{t})-z\big)^{-\frac{1}{2}}\big(\hat B^+_{{W}}(z)\big)^{r} \big(\hat{H}_{q}(\vec{t})-z\big)^{-\frac{1}{2}}dz\Big\|_0\\
%
%
\nonumber \leq &\frac{1}{2\pi}\underset{z \in C_1^+}{max}\|\hat B^+_{{W}}(z)\|_0^{r} \int_{C_{1}^+}\big\|\big(\hat{H}_{q}(\vec{t})-z\big)^{-1} \big\|_0 ds
%
%
\leq  \big(\hat{c}k^{-{\gamma}}\big)^r. 
\end{align}
\end{proof}
%
Let $u_{{W}}^+ , \psi _{{W}}^+$ be defined by Definition \ref{def of u} for $W=W^+$.
%
%
\begin{lemma} \label{conv of psi_m}
Under the  assumptions of Lemma \ref{conv of E^_m} and, for every sufficiently 
large $k>k_{0}(V,\gamma _2, \delta )$ and every $A\in \mathbb{C}:|\sigma||A|^{2}<k^{-{\gamma}_2}$, the function $\psi_{{W}}^+(\vec{x})$ is the limit of the sequence $\psi_m^+(\vec{x})$ in the norm $\|\cdot\|_*$. Moreover,
\begin{align}\label{psi_m -psi}
\|\psi^+_{m}-\psi ^+_{{W}} \|_{*}\leq(\hat{c}k^{-{\gamma}})^{m+1},\  \ m=0,1,...
\end{align}
\end{lemma}
\begin{proof}
Using (\ref{estimate E^_m -E}), we get
\begin{align}
\nonumber \|\psi_m ^+-\psi ^+_{{W}}\|_*\leq & |A|\|\hat{E}^+_m(\vec{t})-\hat E^+_{{W}}(\vec{t})\|_0\\
\leq & |A|(\hat{c}k^{-{\gamma}})^{m+1}. \ \ 
\end{align}
\end{proof}
\begin{corollary} \label{conv of u_m}
The sequence $u ^+_{{m}}$ converges to $\hat u^+_{{W}}$ in $C(Q)$.
\end{corollary}
%
\begin{corollary}\label{M^W =W} 
The operator $\hat{\mathcal{M}}^+$ maps the operator $W^+$ into $W^+$. In other words, $\hat{\mathcal{M}}^+W^+=W^+$.
\end{corollary}
\begin{proof}
We know that $W_m^+ \rightarrow W^+$ with respect to the norm $\|\cdot\|_*$. We also know that $W_{m+1}^+=\hat{\mathcal{M}}W_m^+$ by the equation (\ref{W_m}). It follows that $\hat{\mathcal{M}}^+W_m^+ \rightarrow W^+$ as $m\rightarrow \infty$. On the other hand,
\begin{align*}
\|\hat{\mathcal{M}}^+W_m^+ - \hat{\mathcal{M}}^+W^+\|_* \leq |\sigma| \|\psi_{m}^+ - \psi ^+\|_* \|\bar{\psi}_{m} ^+ + \bar{\psi}^+\|_*.
\end{align*}
This implies, by (\ref{psi_m -psi}), that $\hat{\mathcal{M}}^+W_m^+ \rightarrow \mathcal{M}^+W^+$ in the norm $\|\cdot\|_*$. Therefore, $\hat{\mathcal{M}}^+W^+=W^+$.
\end{proof}
\begin{remark}
Note that $W^+$ depends on $V$, $\sigma $, $A$, $\vec t $.
\end{remark}
We use the notations $\hat{\lambda}^+_m(\vec{t})$ and $\hat \lambda^+_{{W}}(\vec{t})$ for the eigenvalues corresponding to   $\hat{E}^+_m(\vec{t})$ and $\hat E^+_{{W}}(\vec{t})$, respectively.
%
\begin{lemma} \label{conv of lambda^_m}
Under  assumptions of Lemma \ref{conv of E^_m} and, for every sufficiently large $k>k_{0}(V, \gamma _2, \delta )$ and every $A\in \mathbb{C}:|\sigma||A|^2 <k^{-{\gamma}_2}$, the sequence $\hat{\lambda}^+_m(\vec{t})$ converges to $\hat \lambda^+_{{W}}(\vec{t})$. The limit $\hat \lambda ^+_{W}(\vec{t})$ is defined by
\begin{align}\label{def of lambda_W^}
\hat \lambda ^+_{W}(\vec{t})
=\hat{\lambda}^+(\vec{t}) + 
\sum_{r=1}^{\infty}\hat g^+_{{W},r}(k,\vec t).  
\end{align}
Moreover,
\begin{align}\label{estimate of g_Wr}
|\hat g^+_{{{W}, r}}(k,\vec t)|<\hat c (\hat{c} k^{-{\gamma}})^r, 
\end{align}
where $g^+_{{{W},r}}$ is given by  \eqref{tilde{g}_r} with ${W^+}-P_qVP_q$ instead of $\hat V$:
\end{lemma}
\begin{proof}
%
%

First, we prove (\ref{estimate of g_Wr}). 
 For simplicity, we  denote here $\hat{E}^+(\vec{t})$ by $E_0$ and $E_1=I-E_0$.
Note that, for any $r$,
\begin{align*}
\oint_{C^+_{1}}
\big(E_1\hat{B}^+_{m}(z)E_1\big)^r
dz=0,
\end{align*}
since $E_1\hat{B}^+_{m}(z)E_1$ is holomorphic inside $C_1^+$.
Consider
\begin{align*}
 \mathcal{F}_{m,r}=\oint_{C^+_{1}}
 (\hat{B}^+_{m}(z))^r
 dz.
 \end{align*}
Then,
 \begin{align}
\nonumber \mathcal{F}_{m,r}= &\oint_{C^+_{1}}
\left[(\hat{B}^+_{m}(z))^r-(E_1\hat{B}^+_{m}(z)E_1)^{r}\right](
dz\\
\nonumber = &\oint_{C^+_{1}}
\left[\big((E_0+E_1)\hat{B}^+_{m}(z)(E_0+E_1)\big)^r-(E_1\hat{B}^+_{m}(z)E_1)^{r}\right]
dz\\
\nonumber = &\sum \oint_{C^+_{1}}
E_{i_1}\hat{B}^+_{m}(z)E_{i_2}...E_{i_r}\hat{B}^+_{m}(z)E_{i_{r+1}}
dz,
\end{align}
where the sum is taken over $i_1,...,i_{r+1}=0,1$, and $\exists s$ such that $E_{i_s}=E_0$. Let $E_{i_{r_o}}$ be the first one to be equal to $E_0$. Since $E_0$ belongs to the trace class, it is obvious that $E_{i_1}\hat{B}^+_{m}(z)E_{i_2}...E_{i_r}\hat{B}^+_{m}(z)E_{i_{r+1}}$ is also in the trace class. Hence, we get by (\ref{estimate of B^_s})
\begin{align}
 \|E_{i_1}\hat{B}^+_{m}(z)E_{i_2}...E_{i_r}\hat{B}^+_{m}(z)E_{i_{r+1}}\|_1\leq & \|\hat{B}^+_{m}(z)\|_0^{r_0}\|\hat{B}^+_{m}(z)\|_0^{r-{r_0}}
\leq  (\beta k^{-{\gamma}})^r.\label{norm of EBE}
\end{align}
Thus,
\begin{align}
\nonumber \|\mathcal{F}_{m,r}\|_1= &\left\|\sum \oint_{C^+_{1}}
E_{i_1}\hat{B}^+_{m}(z)E_{i_2}...E_{i_r}\hat{B}^+_{m}(z)E_{i_{r+1}}
dz\right\|_1\\
 \leq & \underset{z\in C^+_1}{max}\|E_{i_1}\hat{B}^+_{m}(z)E_{i_2}...E_{i_r}\hat{B}^+_{m}(z)E_{i_{r+1}}\|_1 \int_{C_1^+}\
ds
 \leq  2\pi( \beta k^{-\gamma })^r\label{estimate of F_mr}
\end{align}
and $\mathcal{F}_{m,r}$ converges  to 
\begin{align}\label{def of F_r}
  \mathcal{F}_{r}=\oint_{C^+_{1}}
  B^+_{{W}}(z)^r
  dz  
\end{align} in the trace class. 
Let 
$\hat{g}^+_{m, {r}}$ be given by  \eqref{tilde{g}_r} with $W_m^+-P_qVP_q$ instead of $\hat V$.
Since $\hat{g}_{m,r}(k,\vec{t})= \frac{(-1)^r}{2\pi ir}\mbox{Tr }\mathcal{F}_{m,r}$, it follows that $\hat{g}_{m,r}(k,\vec{t})$ converges to $\hat g_{W,r}(k,\vec{t})$. Then,
\begin{align}
   |\hat g_{{W},r}(k,\vec t)|=& \left|\frac{(-1)^r}{2\pi ir}Tr \mathcal{F}_{r}\right|
  \leq  \frac{r^{-1}}{2\pi}\cdot 2\pi(\beta\|V\|_*k^{-{\gamma}})^r 
 \leq  \hat c (\hat{c}\|V\|_*k^{-{\gamma}})^r.
\end{align}
This proves (\ref{estimate of g_Wr}).
Now, it is clear that $\lambda^+_{{W}}(\vec{t})$ is the limit of $\{\hat{\lambda}^+_m(\vec{t})\}_{m=0}^{\infty}$ since
\begin{align}
\nonumber \left|\hat{\lambda}^+_m(\vec{t})-\hat \lambda ^+_{{W}}(\vec{t})\right|  
%
\nonumber \leq &\sum_{r=1}^{\infty}\left|\hat{g}_{m,{r}}(k,\vec{t})-g_{{W},r}(k,\vec{t})\right|.
 \end{align} 
\end{proof}
The following theorem is the main result of the paper for the case $2l>n$. 
\begin{theorem} \label{mainmain}
 Let $\gamma _2>0$, $0<9\delta <2l-n $. Then, for  each  sufficiently large $k$: $k>k_{0}(V, \gamma _2, \delta)$, the following holds. If $\vec{t}$ belongs to the $(k^{-n+1-7\delta})$-neighborhood in $K$ 
of the resonant set $\chi_{q}(k,\delta)$ and  $A\in \mathbb{C}: |\sigma||A|^{2}<k^{-{\gamma}_2}$, then, there is a pair of functions $u^{\pm }(\vec{x}, \vec t)$ and the corresponding  real values 
$\lambda^{\pm}(\vec{t})$, 
 solving the  equation \eqref{main equation}
with the boundary conditions \eqref{main condition}. Moreover, the following is true as $k\to \infty $:
\begin{align}\label{u and lambda 2l>n}
u^{\pm}(\vec{x},\vec t) = & Ae^{i\langle{k,\vec{x}}\rangle}\big(\psi_{-1}^{\pm}(\vec{x},\vec t)+\phi ^{\pm }(\vec{x},\vec t)\big),\\ 
\lambda ^{\pm }(\vec{t}) = & {\hat{\lambda}}^{\pm }(\vec t)+O\left(k^{-{\gamma}}\right),
\end{align}
where $\psi_{-1}^{\pm }$ is given by (\ref{def of psi_-1}) and Definition \ref{d} (see also \eqref{eigenvector of lambda +}, \eqref{eigenvector of lambda -}) and $\phi (\vec{x},\vec t)$ is periodic in $\vec{x},$  and satisfies:
\begin{align}\label{estimate of u^}
\|\phi \|_* \leq \hat c k^{-{\gamma}},
\end{align}
  ${\gamma}=\min \{\gamma _1, \gamma _2 \}$, $2\gamma _1=2l-n-8\delta $.
\end{theorem}

From now on we set
\begin{equation} \label{simplify}
u^{\pm}(\vec{x},\vec t)=u^{\pm}_{{W}}(\vec{x},\vec t), \ \ \ \lambda  ^{\pm }(\vec{t})=\hat \lambda ^{\pm }_{{W}}(\vec t), \ \ \ \ E^{\pm }=\hat E^{\pm }_{W}.
\end{equation}

\begin{proof}
We note that the function $u^{\pm}_{{W}}$ defined by (\ref{def of u eq-e}) and the value $\lambda_{{W}}^{\pm }(\vec{t})$ given in Lemma \ref{conv of lambda^_m} by formula (\ref{def of lambda_W^}) solve the equation:
\begin{align}\label{sol u, lambda_W^, 2l>n}
\big(H_0(\vec{t}) + {W}^{\pm}\big)u^{\pm}_{{W}}(\vec{x},\vec t)=\hat \lambda ^{\pm }_{{W}}(\vec t) u^{\pm}_{{W}}(\vec{x},\vec t), ~ \vec{x}\in Q,\ \ W^{\pm}=W^{\pm}(\vec x,\vec t, A),
\end{align}
and  satisfying the boundary conditions (\ref{main condition}). Using Corollary \ref{M^W =W}, we can rewrite equation (\ref{sol u, lambda_W^, 2l>n})  as  \eqref{main equation}. Considering \eqref{simplify} we finish the proof.

\end{proof}

\subsection{The Differentiability of the Eigenvalue and Its Spectral Projection}

\begin{lemma} \label{22.12} Under conditions of Lemma \ref{estimates}

\begin{align}\label{estimate of W_m difference+}
    \|\nabla_{\vec t}(W_m^{\pm }-W_{m-1}^{\pm })\|_* \leq k^{2l-1}(\hat{c}k^{-{\gamma}})^m,
\end{align}
\begin{align}\label{estimate of V-W_m+}
     \|\nabla_{\vec t}(W_{m}^{\pm }-V)\|_* \leq k^{2l-1}\sum_{r=1}^{m}(\hat{c}k^{-{\gamma}})^{r},
\end{align}
\begin{align}\label{estimate of E^_0 difference++}
\|\nabla_{\vec t}\left (\hat{E}_{0}^{\pm }(\vec{t})-\hat{E}^{\pm }_{-1}(\vec{t})\right)\|_0\leq \hat{c}k^{n-1+7\delta -{\gamma}},
\end{align}
\begin{align}\label{estimate of E^_m difference++}
\|\nabla_{\vec t}\left( \hat{E}^{\pm }_{m-1}(\vec{t})-\hat{E}^{\pm }_{m-2}(\vec{t})\right)\|_0\leq k^{2l-1}(\hat{c}k^{-{\gamma}})^{m-1}, \ m \geq 2.
\end{align}
\end{lemma}
\begin{proof}
The inequality \eqref{estimate of E^_0 difference++} follows from the linear case inequality \eqref{2.2.35}. To obtain the estimates \eqref{estimate of W_m difference+}, \eqref{estimate of V-W_m+}  and \eqref{estimate of E^_m difference++}, we use the obvious inequality:
$$\left\|\nabla _{\vec t}\left(\hat H_q(\vec t)-z\right)^{-1/2}\right\|_0<ck^{2l-1}d^{-3/2},\ \ \ z\in C_1^{\pm} $$   Further considerations are analogous to that in  Lemma \ref{estimates}.
\end{proof}


\begin{lemma}\label{conv of E^_m'} Under conditions of Lemma \ref{conv of E^_m} the sequences $\{\nabla _{\vec t}\hat{E}_m^{\pm }(\vec{t})\}_{m=0}^{\infty}$ converge to 
 $\nabla_{\vec t}E ^{\pm }_{{}}(\vec{t})$  in $\|\cdot\|_0$ and
\begin{align}\label{estimate E^_m -E-1}
\left\|\nabla _{\vec t}\left(\hat{E}^{\pm }_m(\vec{t})- E^{\pm }_{{}}(\vec t)\right)\right\|_0\leq k^{2l-1}(\hat{c}k^{-{\gamma}})^{m+1}. 
\end{align} 
The operator $\nabla_{\vec t} E^{\pm }_{{}}(\vec{t})$ is given as
\begin{align}\label{def of E_W'-1}
 \nabla _{\vec t}E^{\pm }_{}(\vec{t})=  \nabla _{\vec t}\hat{E}^{\pm}(\vec{t}) + \sum_{r=1}^{\infty} \nabla _{\vec t}\hat G^{\pm}_{{W},r}(k,\vec{t}),
\end{align}
where
\begin{align}\label{G_W^ estimate+}
\| \nabla _{\vec t}\hat G^{\pm}_{{W},r}(k,\vec{t})\|_0\leq k^{2l-1}\left(\hat{c}k^{-{\gamma}}\right)^{r}.
\end{align}
The sequences $\nabla _{\vec t}\hat{\lambda}^{\pm}_m(\vec{t})$ converge to $\nabla _{\vec t} \lambda ^{\pm}_{{}}(\vec{t})$. The limits $\nabla _{\vec t}\lambda ^{\pm}_{}(\vec{t})$ are defined by
\begin{align}\label{def of lambda_W}
\nabla _{\vec t}\lambda ^{\pm}_{}(\vec{t})
=\nabla _{\vec t}\hat{\lambda}^{\pm}(\vec{t}) + 
\sum_{r=1}^{\infty}\nabla _{\vec t}\hat g^{\pm}_{{W},r}(k,\vec t),
\end{align}
where
\begin{align} \label{estimate of gWr}
|\nabla _{\vec t}\hat g^{\pm}_{{{W}, r}}(k,\vec t)|<k^{2l-1}(\hat{c}k^{-{\gamma}})^r.
\end{align}

\end{lemma}
\begin{proof} The proof is analogous to those of Lemma  \ref{conv of E^_m} and \ref{conv of lambda^_m} up to differentiation, Lemma \ref{22.12} being used. \end{proof}

The next theorem easily follows from the previous lemma.

\begin{theorem} \label{coefficients of g and G_W^}  Under conditions of  Theorem \ref{mainmain}:
\begin{align}\label{T(m) of diff of spectral proj}
\| \nabla _{\vec t}\left(E^{\pm }_{}(\vec{t})-\hat{E}^{\pm }(\vec{t})\right)\|_0 <C(V)k^{2l-1-{\gamma}} ,
\end{align}
\begin{align}\label{nabla lambda}
\left|\nabla _{\vec t}\big(\lambda ^{\pm }_{}(\vec{t})-\hat{\lambda}^{\pm }(\vec{t})\big)\right|<C(V)k^{2l-1-{\gamma}}.
\end{align}
\end{theorem}
\begin{corollary}
\begin{align}\label{nabla lambda-1}
\left|\nabla_{\vec t} \lambda ^{\pm }_{}(\vec{t})\right|=2lk^{2l-1}\big(1+O(k^{-\gamma })+O(k^{-1})\big).
\end{align}
\end{corollary}
The corollary follows from the theorem and the obvious relation  $\nabla \hat{\lambda}^{\pm}(\vec t) = 2lk^{2l-1}(1+O(k^{-1}))$. 

Let us consider the surface $\hat{\lambda}^{\pm}(\vec t)=\lambda _0$ for a fixed $\lambda _0$, $\lambda _0>k_0^{2l}(V,\gamma _2,\delta )$. Note that  the parts $\hat{\lambda}^{+}(\vec t)=\lambda _0$  and $\hat{\lambda}^{-}(\vec t)=\lambda _0$ do not intersect, since $ v_q\neq 0$.
Thus the deviation of the surface $\hat{\lambda}^{\pm}(\vec t)=\lambda _0$ from the unperturbed one (V=0) is essential. The next theorem follows.

\begin{theorem} \label{25} If $\lambda _0>k_0^{2l}(V,\gamma _2,\delta )$, then the surface $ \lambda ^{\pm }(\vec{t})=\lambda _0 $ is in the $C(V)\lambda _0^{-\hat \gamma }$-neighborhood of  $ \hat{\lambda}^{\pm}(\vec t)=\lambda _0$ for every sufficiently large $\lambda _0$, here $\hat \gamma =
(2l-1+\gamma )(2l)^{-1}$.

\end{theorem}
\begin{corollary} \label{25a} The surfaces ${\lambda}^{+}(\vec t)=\lambda _0$  and $\lambda ^{-}(\vec t)=\lambda _0$ do not intersect and located at the distance greater than $|v_q|+O\left(\lambda _0^{-\hat \gamma }\right)$  from each other as $\lambda _0\to \infty $.
\end{corollary}
\section{Solutions of Nonlinear Scr\"{o}dinger Equation in Dimension Two.}
In this section, we present  resonant solutions of  (\ref{main equation}), \eqref{main condition} for $n=2$, $l=1$. 
 The equation is 
\begin{equation}\label{GPE n=2}
-\Delta u+Vu+\sigma |u|^{2}u=\lambda u.
\end{equation}
The proof of the result is analogous to that for $2l>n$.
Indeed, let \
$S(k,\epsilon )$ be given by \eqref{S-q} up to replacing of  $4k^{-n+2-\delta}$ by $\epsilon $:
\begin{align} \label{S-q'}
S_q(k,\epsilon)=\{\vec{x}\in S(k):\big||\vec{x}|^2-|\vec{x}-\vec{P}_q(0)|^2\big| <\epsilon\}.
\end{align}

We first state the  geometric lemma.
\begin{lemma}\label{difference eigenvalues 2l=n non-lin}
Let  $0<\epsilon < \varepsilon _0$. Then, 
for sufficiently large$k$,  $k>k_1(q,  \varepsilon _0)$, there exists a resonant set $\chi _{{q}}(k,\epsilon)\subset \mathcal{K}S_{{q}}(k,\epsilon)$ such that, for any  $\vec t \in \chi _{{q}}(k,\epsilon)$, the followings hold:
\begin{align}
&1.~\mbox{there exists a unique 
$j\in \mathbb{Z}^{n}$ such that $|\vec{P}_j(\vec{t})| =k$},\\
&2.~\mbox{ }|p_{j}^2(\vec{t})-p_{j-q}^2(\vec{t})|<\epsilon,\\
&3.~\min_{m\neq j,j-{q}}|p_{j}^{2}(\vec{t})-p_{m}^{2}(\vec{t})|>2\epsilon^6.
\end{align}
Moreover, for any $\vec{t}$ in the $(k^{-1}\epsilon^7)$-neighborhood of $\chi_{{q}}(k,\epsilon)$ in $\mathbb{C}^2$, there exists a unique $j\in\mathbb{Z}^2$ such that $|p_{j}^{2}(\vec{t})-k^2|$ $< 5\epsilon^7$ and the second  and third conditions above are satisfied. 

The set $\chi_{{q}}(k,\epsilon)$ has an asymptotically full measure on $\mathcal{K}S_{{q}}(k,\epsilon)$ as $\epsilon \to 0$, that is
\begin{equation}
\frac{s(\mathcal{K}S_{q}(k,\epsilon)\setminus\chi_{q}(k,\epsilon))}{s(\mathcal{K}S_{q}(k,\epsilon))}<c{\epsilon},\ \ c\neq c(k).
\end{equation}
\end{lemma}

%
For all of the followings we assume that 

\begin{equation} \|V\|_*<\epsilon^9,\ \  \  |v_{q}|>\epsilon^{10 },\  \  \  |\sigma ||A^2|<\epsilon ^ {11} . \label{restrictions} \end{equation}

\begin{corollary}\label{geometric cor, nonlin, 2l=n}
If $\vec{t}$ belongs to the  $(\frac{1}{8}k^{-1}\varepsilon ^{10})$- neighborhood of $\chi_{{q}}(k,\epsilon)$, then for all z on the circle $$C_1^+=\{z: |z-\hat{\lambda}^+(\vec{t})|=d\}, \ \ \ d=\frac{1}{3}\epsilon ^{10},$$ both of the following inequalities are true:
\begin{align} \label{oh}
2|p_m^{2}(\vec{t})-z| \geq & \epsilon^6, ~m\neq j,j-q,
\end{align}
\begin{align} \label{ah}
|\hat{\lambda}^{\pm} (\vec{t})-z| \geq \frac{1}{12}\epsilon^{10} , 
\end{align}
$\hat{\lambda}^{\pm} $ being the eigenvalues of \eqref{2x2}.
\end{corollary}

Now, consider the map $\hat{\mathcal{M}}$ defined by (\ref{map M}). The following lemma can be proved by analogy with Lemma \ref{estimates}. 
\begin{lemma}\label{estimates two dim} 

There is $\epsilon _0$, $0< \epsilon _0\neq \epsilon _0(\lambda )$, such that for any $0<\epsilon <\epsilon _0$ under the conditions \eqref{restrictions} and  for any  sufficiently large $k$: $k>k_{1}(V,\epsilon _0)$, the following holds.
Let $\vec t$ belong to the $(\frac{1}{8}k^{-1}\varepsilon ^{10})$-neighborhood in $K$ of the set $\chi _{q}(k,\epsilon)$.  Then, for any $ m=1, 2, ... $:

%
\begin{align}\label{estimate of W_m difference two dim-e}
\|W_m^+-W^+_{m-1}\|_* \leq \epsilon ^{11} ({c}\epsilon)^{m-1},
\end{align}
\begin{align}\label{estimate of V-W_m two dim-e}
 \|W^+_{m}-V\|_* \leq \epsilon ^{11}\sum_{r=1}^{m}({c}\epsilon)^{r-1},
\end{align}
\begin{align}\label{estimate of E^_m difference two dim-e}
\|\hat{E}^+_{0}(\vec{t})-\hat{E}^+_{-1}(\vec{t})\|_0\leq 
c\epsilon ,
\end{align}

\begin{align}\label{estimate of E^_m difference two dim-e+}
\|\hat{E}^+_{m-1}(\vec{t})-\hat{E}^+_{m-2}(\vec{t})\|_0\leq 
(c\epsilon )^ {m-1},
\end{align}
where $c$ is an absolute constant.
\end{lemma}

\begin{corollary}\label{Cauchy seq cor two dim}
The sequence $\{W^+_m\}_{m=0}^{\infty}$ converges to a continuous and periodic function $W$ with respect to the norm $\|\cdot\|_{*}$. The following estimate holds:
\begin{align}\label{diff W-W_m two dim}
\|W^+-W_m^+\|_*\leq 2\epsilon ^{11} (c\epsilon )^{m}.
\end{align}
\end{corollary}

\begin{proof} The following facts can be easily checked:
\begin{align}
&1. \|(I-P_q)(\hat{H}_q(\vec{t})-z)^{-\frac{1}{2}}\|_0=\|(\hat{H}_q(\vec{t})-z)^{-\frac{1}{2}}(I-P_q)\|_0 <2\epsilon^{-3}, \label{Dec24a}\\
&2. \|P_q(\hat{H}_q(\vec{t})-z)^{-\frac{1}{2}}\|_0=\|(\hat{H}_q(\vec{t})-z)^{-\frac{1}{2}}P_q\|_0 <c\epsilon^{-5}, \label{Dec24b}
\end{align}
$P_q$, $\hat H_q$ being defined by  \eqref{def P_q}, \eqref{H_q}.
Further we use an induction. For the first step we need to show that 
\begin{equation} \label{base case of induct 1-e}
\|\hat E^+_0(\vec t)-\hat E^+_{-1}(\vec t)\|_0\leq c\epsilon ,
\end{equation}
\begin{equation}\label{base case of induct 2-e}
\|W^+_1-W_0^+\|_*\leq c\epsilon ^{11}.
\end{equation}
We start with \eqref{base case of induct 1-e}. This is a perturbative formula for a linear operator. We use the series \eqref{eval in singular set},  \eqref{proj in singular set}. Indeed,  we take  $\hat{B}_0(z)$ given by 
\begin{equation} \label{B^_0 def-e}
\hat{B}_0(z)=\left(\hat{H}_q(\vec{t})-z\right)^{-\frac{1}{2}}
\hat{W}_0
\left(\hat{H}_q(\vec{t})-z\right)^{-\frac{1}{2}}, \ \ \ \hat W_0=V-P_qVP_q.
\end{equation}
Using  $\|V\|_*<\epsilon^9$ and  \eqref{Dec24a}, \eqref{Dec24b}, we obtain:
\begin{align}
\underset{z\in C_1^+}{max}{\|\hat{B}_0(z)}\|_0
< c\epsilon .\label{11-e}
\end{align}
Now, we consider $\hat{G}^+_{0,r}(k,\vec{t})$, see \eqref{tilde{G}_r}.            
 Applying (\ref{11-e}) we get:
\begin{align}\label{estimate of G^_0,r+}
\|\hat{G}^+_{0,r}(k,\vec{t})\|_0\leq \big(c\varepsilon \big)^r  .  
\end{align}
Hence,  \eqref{base case of induct 1-e} follows. 
Let us estimate $\|W^+_{1}-W_{0}^+\|_{*}$. We use again Definition \ref{def4} and
  \eqref{W_1 - W_0}--\eqref{estimate of sum of psi}.
Applying  the condition $|\sigma A|^2<\epsilon ^{11}$ we obtain \eqref{base case of induct 2-e}.

Now, let us assume that (\ref{estimate of W_m difference two dim-e}) -\eqref{estimate of E^_m difference two dim-e+}  are satisfied for all $s=1,2,...,m-1$, i.e.,
\begin{align}\label{estimate of W_s difference-ee}
    \|W_s^+-W_{s-1}^+\|_* \leq  \epsilon ^{11} ({c}\epsilon)^{s-1}
\end{align}
\begin{align}\label{estimate of V-W_s-ee}
     \|W^+_{s}-V\|_* \leq \epsilon ^{11}\sum_{r=1}^{s}({c}\epsilon)^{r-1},
\end{align}
\begin{align}\label{estimate of E^_s difference-ee}
\|\hat{E}^+_{s-1}(\vec{t})-\hat{E}^+_{s-2}(\vec{t})\|_0\leq ({c}\epsilon)^{s-1}.
\end{align}
First, let us prove (\ref{estimate of E^_m difference two dim-e+}). 
We define $\hat{B}^+_{m-1}(z)$ by
\eqref{B^_s def}. Considering as in \eqref{new},  \eqref{21.12}, we get:
\begin{align} \label{21.12-e}
\underset{z\in C_1}{max}{\|\hat{B}^+_{m-1}(z)}\|_0\leq \|\hat{W}^+_{m-1}\|_0\epsilon ^{-8}
+d^{-1}\|P_q(W^+_{m-1}-V)P_q\|_0 <c\epsilon .
\end{align}
 It follows $\|\hat{E}^+_{m-1}(\vec{t})\|_0\leq 3$.
Now, we consider \eqref{diff of B_s} and further. The analog of \eqref{estimate diff of B_s} is
\begin{align}\label{estimate diff of B_s-aa}
\|\hat{B}^+_{m-1}(z)-\hat{B}^+_{m-2}(z)\|_0
\leq &d^{-1}\|W^+_{m-1}-W^+_{m-2}\|_0, \ \ m\geq 2.
\end{align}
The analog of \eqref{R} is
\begin{align}
\nonumber R\leq   &c\sum_{r=1}^{\infty}d^{-1}\|W^+_{m-1}-W^+_{m-2}\|_0\big(c\epsilon \big)^{r-1}
%
%
 \leq &cd^{-1}\|W^+_{m-1}-W^+_{m-2}\|_0
\leq c \epsilon (c\epsilon )^{m-2},\label{R-e}
\end{align}
 $m\geq 2$. Thus, \eqref{estimate of E^_m difference two dim-e+} follows.
Next, we consider $\|W^+_m-W^+_{m-1}\|_0$.  Using the first line of \eqref{induct step}, we obtain:

\begin{equation}
 \label{Ch}
 \|W^+_m-W^+_{m-1}\|_0<\epsilon ^{11}(c\epsilon )^{m-1}.
\end{equation}
Therefore, (\ref{estimate of W_m difference two dim-e}) and  (\ref{estimate of V-W_m two dim-e}) hold.

\end{proof}
The convergence of the sequences of the eigenvalues and their spectral projections follow.
\begin{lemma}\label{two dim conv of E^_m} There is $\epsilon _0$, $0< \epsilon _0\neq \epsilon _0(\lambda )$, such that for any $0<\epsilon <\epsilon _0$ under the conditions \eqref{restrictions} and for any sufficiently large $k$, $k>k_{1}(V,\epsilon _0)$, the following holds.
Let $\vec t$ belong to the $(\frac{1}{8}k^{-1}\varepsilon ^{10})$-neighborhood in $K$ of the set $\chi _{q}(k,\epsilon)$. Then, 
the sequence $\{\hat{E}_m^+(\vec{t})\}_{m=0}^{\infty}$ converges with respect to the norm $\|\cdot\|_0$ to a one-dimensional spectral projection $\hat E^+_{{W}}(\vec{t})$ of ${H}_0(\vec{t})+{W}^+$ and
\begin{align}\label{two dim estimate E^_m -E}
\|\hat{E}^+_m(\vec{t})-\hat E^+_{W}(\vec t)\|_0<({c}\epsilon)^{m+1},\ \  m\geq 0,
\end{align} 
\begin{align}\label{two dim estimate E^_m -E-1}
\|\hat{E}^+_{-1}(\vec{t})-\hat E^+_{W}(\vec t)\|_0<{c}\epsilon^{}.
\end{align} 
The projection $\hat E^+_{{W}}(\vec{t})$ is given by \eqref{def of E_W^}
where $\hat{E}^+$ is the spectral projection of $\hat{H}_q(\vec{t})$ corresponding to the eigenvalue $\hat{\lambda}^+(\vec{t})$ and $\hat G^+_{W,{r}}(k,\vec{t})$ is defined by \eqref{tilde{G}_r}  with  ${W}^+-P_qVP_q$ instead of $\hat V$.
Moreover, the following estimate is valid:
\begin{align}\label{two dim G_W^ estimate}
\|\hat G^+_{{W},r}(k,\vec{t})\|_0\leq (\hat{c}\epsilon)^{r}.
\end{align}
\end{lemma}
\begin{proof} The proof of the lemma is analogous to that of Lemma \ref{conv of E^_m}.\end{proof}
Using  Definition \ref{def of u}, we obtain the following results analogous to Lemma \ref{conv of psi_m}, Corollary \ref{conv of u_m}, and Corollary \ref{M^W =W}.
\begin{lemma} \label{two dim conv of lambda^_m}
 There is $\epsilon _0$, $0< \epsilon _0\neq \epsilon _0(\lambda )$, such that for any $0<\epsilon <\epsilon _0$ under the conditions \eqref{restrictions}  and  for any sufficiently large $k$, $k>k_{1}(V,\epsilon _0),$ the following holds.
Let $\vec t$ belong to the $(\frac{1}{8}k^{-1}\varepsilon ^{10})$-neighborhood in $K$ of the set $\chi _{q}(k,\epsilon)$. Then, 
the sequence $\hat{\lambda}^+_m(\vec{t})$ converges to $\lambda ^+_{{W}}(\vec{t})$ given by (\ref{def of lambda_W^}), and for all $r$ we have
\begin{align}
|\hat g^+_{{{W},r}}(k,\vec t)|< \epsilon ^{10}({c}\epsilon)^r,  
\end{align}
where $\hat g_{W,r}$ is given by (\ref{tilde{g}_r}) with  ${W}^+-P_qVP_q$ instead of $\hat V$.
\end{lemma}
\begin{corollary} \label{conv of u_m+}
The sequence $u ^+_{{m}}$ converges to $\hat u^+_{{W}}$ in $C(Q)$.
\end{corollary}
%
\begin{corollary}\label{M^W =W+} 
The operator $\hat{\mathcal{M}}^+$ maps the operator $W^+$ into $W^+$. In other words, $\hat{\mathcal{M}}^+W^+=W^+$.
\end{corollary}

Next, we present a solution of non-linear Scr\"{o}dinger equation  in the dimension two.
\begin{theorem}  \label{main theorem 2l=n} There is $\epsilon _0$, $0<\epsilon _0\neq \epsilon _0(\lambda )$, such that for any $0<\epsilon <\epsilon _0$ under conditions \eqref{restrictions}  and for any sufficiently large $k$, $k>k_{1}(V,\epsilon _0)$, the following holds.
Suppose $\vec{t}$ belongs to the $(\frac{1}{8}k^{-1}\varepsilon ^{10})$-neighborhood in $K$ 
of  the resonant set $\chi_{q}(k,\epsilon )$. Then, there is  a pair of functions $u^{\pm}(\vec{x},\vec t)$ and the corresponding  real values $\lambda ^{\pm}_{}(\vec{t})$, satisfying the  equation 
\begin{align}\label{solution for dim 2}
-\Delta u^{\pm}+V(\vec{x})u^{\pm}+\sigma |u^{\pm}|^{2}u^{\pm}=\lambda ^{\pm}u^{\pm},~\vec{x}\in Q,
\end{align}
and  the quasi-periodic boundary condition (\ref{main condition}). The followings hold:
\begin{align} \label{solution eq 2l=n}
u^{\pm}(\vec{x},\vec t)= & A e^{i\langle{k,\vec{x}}\rangle}\left(\psi_{-1}^{\pm}(\vec{x},\vec t)+\phi^{\pm}(\vec{x}, \vec t)\right),\\ 
\lambda^{\pm}(\vec{t})= & \hat{\lambda}^{\pm}(\vec t)+O\left( \epsilon ^{11}\right),
\end{align} 
where $\psi_{-1}^{\pm}$ is as defined by (\ref{def of psi_-1}) and Definition \ref{d} (see also \eqref{eigenvector of lambda +}, \eqref{eigenvector of lambda -}) and $\phi ^{\pm}(\vec{x},\vec t)$ is periodic in $\vec{x},$  and satisfies:
  \begin{equation}
\|\phi^{\pm}\|_{*}\leq \epsilon. \end{equation}
\end{theorem}
The proof of the theorem is similar to that of Theorem \ref{mainmain}, the notation \eqref{simplify} being used.
\begin{theorem} \label{coefficients of g and G_W^+} 
Under the conditions of Theorem \ref{main theorem 2l=n}  the series (\ref{def of E_W^}) and (\ref{def of lambda_W^})  can be differentiated with respect to $\vec{t}$ retaining their asymptotic character. Moreover, the coefficients 
$\hat g^{\pm}_{W,r}(k,\vec{t})$ and $\hat G^{\pm}_{{W,r}}(k,\vec{t})$ satisfy the following estimates in the $(\frac{1}{8}k^{-1}\varepsilon ^{10})$-neighborhood in $\mathcal{C}^2$ of the set $\chi_{q}(k,\delta )$:
\begin{align}\label{T(m) of g_W^}
\mid \nabla _{\vec t} \hat g^{\pm}_{{W},r}(k,\vec{t})\mid <(\hat{c}\varepsilon )^{r}k^{}
\end{align}
\begin{align}\label{T(m) of G_W^}
\|\nabla _{\vec t} \hat G^{\pm}_{{W},{r}}(k,\vec{t})\|_0< \left({c}\varepsilon\right)^{r}(k \varepsilon ^{-10}).
\end{align}
\end{theorem}
\begin{corollary} \label{derivatives of lambda_W^ and E_W^} The followings hold for the perturbed eigenvalue and its
spectral projection:
\begin{align}\label{derivatives of lambda_W^}
\big|\nabla _{\vec t} \big(\lambda ^{\pm}_{}(\vec{t})-\hat{\lambda}^{\pm }(\vec{t})\big)\big| <C(V)k \varepsilon ^{}
\end{align}
\begin{align}\label{derivatives of E_W^} 
\| \nabla _{\vec t} (E^{\pm}_{}(\vec t)-\hat{E}^{\pm})\|_0 <C(V)k \varepsilon ^{-9}.
\end{align}
\end{corollary}
\begin{corollary}
\begin{equation} \label{12.26} |\nabla _{\vec t} \lambda ^{\pm }_{}(\vec{t})|=2k(1+O(\epsilon )).\end{equation}
\end{corollary}

Let us consider the surface $\hat{\lambda}^{\pm}(\vec t)=\lambda _0$ for a fixed $\lambda _0$, $\lambda _0>k_1^{2}(V,\epsilon _0)$. Note that  the parts $\hat{\lambda}^{+}(\vec t)=\lambda _0$  and $\hat{\lambda}^{-}(\vec t)=\lambda _0$ do not intersect, since $ v_q\neq 0$.
Thus the deviation of the surface $\hat{\lambda}^{\pm}(\vec t)=\lambda _0$ from the unperturbed one (V=0) is essential. The next theorem follows.

\begin{theorem} \label{25+} If  $0<\varepsilon <\epsilon _0$, $\lambda _0>k_1^{2l}(V,\epsilon _0)$, then the curves $ \lambda ^{\pm }(\vec{t})=\lambda _0 $  are in the $C(V)\varepsilon ^{11} \lambda _0^{-1/2}$-neighborhood of  the curves $\hat{\lambda}^{\pm}(\vec t)=\lambda _0$. \end{theorem}

\begin{corollary} \label{25a+} The curves ${\lambda}^{+}(\vec t)=\lambda _0$  and $\lambda ^{-}(\vec t)=\lambda _0$ do not intersect and located at the distance greater than $|v_q|+O\left(\varepsilon ^{11} \lambda _0^{-1/2}\right)$  from each other as 
$\varepsilon  \lambda _0^{-1/2}\to 0 $.
\end{corollary}





\end{document}